\documentclass[preprint,2]{elsarticle}
\biboptions{longnamesfirst,square,sort,comma,numbers}
\usepackage{a4wide}
\usepackage[english]{babel}
\usepackage[IL2]{fontenc}
\usepackage[cp1250]{inputenc}
\usepackage[bookmarks=true,unicode=true,plainpages=false,colorlinks=true,linkcolor=black]{hyperref}
\hypersetup{citecolor=blue}
\usepackage{url}
\usepackage{amsmath,amsfonts,amssymb,amsthm}
\usepackage{graphicx}
\usepackage{indentfirst}
\usepackage{mathtools}
\usepackage{color}
\usepackage{float}
\usepackage{caption}
\usepackage{subcaption}
\usepackage{titlesec}

\journal{Nonlinear Analysis: Theory, Methods $\&$ Applications}

\newtheorem{remark}{Remark}[section]
\newtheorem{corollary}{Corollary}[section]
\newtheorem{lemma}{Lemma}[section]
\newtheorem{observation}{Observation}[section]
\newtheorem{theorem}{Theorem}[section]
\theoremstyle{definition}
\newtheorem{definition}{Definition}[section]
\newtheorem{notation}{Notation}[section]

\begin{document}
	\renewcommand*{\today}{July 17, 2018}
	\pagestyle{plain}
	\relpenalty=9999
	\binoppenalty=9999
	
	\title{Unilateral sources and sinks of an activator in reaction-diffusion systems exhibiting diffusion-driven instability}
	
	\author[ZCU FAV]{Martin Fencl}
	\ead{fenclm37@kma.zcu.cz}
	\author[ZCU FAV,MU AV CR]{Milan Ku{\v c}era}
	\ead{kucera@math.cas.cz}
	\address[ZCU FAV]{Dept. of Mathematics, Faculty of Applied Sciences, University of West Bohemia in Pilsen, Univerzitn\'{i} 8, 30614 Plze\v{n}, Czech Republic}
	\address[MU AV CR]{Institute of Mathematics, Czech Academy of Sciences, {\v Z}itn{\' a} 25, 11567 Prague 1, Czech Republic}

	\begin{abstract}
		A reaction-diffusion system exhibiting Turing's diffusion driven 
		instability is considered. The equation for an activator is supplemented by 
		unilateral terms of the type $s_{-}(\mathbf{x})u^{-}$, $s_{+}(\mathbf{x})u^{+}$ describing sources 
		and sinks active only if the concentration decreases below and increases 
		above, respectively, the value of the basic spatially constant solution which is shifted to zero. 
		We show that the domain of diffusion parameters in which spatially 
		non-homogeneous stationary solutions can bifurcate from that constant 
		solution is smaller than in the classical case without unilateral terms. It is 
		a dual information to previous results stating that analogous terms in the equation for an inhibitor imply the existence of bifurcation points 
		even in diffusion parameters for which bifurcation is excluded without unilateral sources. The case 
		of mixed (Dirichlet-Neumann) boundary conditions as well as that of pure 
		Neumann conditions is described.
	\end{abstract}

	\maketitle		
	\section{Introduction}
	\noindent Let's consider a reaction-diffusion system
	\begin{equation}\label{Eq:general model}
	\begin{aligned}
	\frac{\partial u}{\partial t}&= d_{1}\Delta u + f(u,v) + {\tilde{f}_{-}(\mathbf{x},u^{-}) - \tilde{f}_{+}(\mathbf{x},u^+}),\\
	\frac{\partial v}{\partial t}&= d_{2}\Delta v + g(u,v) \quad\text{in}\ \Omega\times\left[0,+\infty\right) 
	\end{aligned}		
	\end{equation}
	where $\Omega\subset\mathbb{R}^{N}$ is a bounded domain with Lipschitz boundary, 
	$d_{1}$ and $d_{2}$ are positive parameters (diffusion coefficients),
	$f,g:\mathbb{R}\times \mathbb{R} \to \mathbb{R}$ are real differentiable functions, $\tilde{f}_{-},\tilde{f}_{+}:\Omega \times \mathbb{R} \to \mathbb{R}$ are functions satisfying Carathéodory conditions and such that there exist 
	\begin{equation}\label{Eq:tilda terms cond}
	\begin{aligned}
	s_{-}(\mathbf{x}):=\frac{\partial \tilde{f}_{-}}{\partial \xi}(\mathbf{x},\xi)|_{\xi=0}\ge 0,\quad 
	s_{+}(\mathbf{x}):=\frac{\partial \tilde{f}_{+}}{\partial \xi}(\mathbf{x},\xi)|_{\xi=0}\ge 0 \quad\text{for a.a. } \mathbf{x}\in \Omega,
	s_\pm \in L_\infty (\Omega).	
	\end{aligned}	
	\end{equation}
	As usually, $u^{+} = \max\{u , 0 \}$ and $u^{-} = \max\{-u , 0 \}$ denotes the positive and negative, respectively, part of $u$.
	We will always assume that
	\begin{equation}\label{Eq:zero fixed point assumption}
	f(0,0)=g(0,0)=\tilde{f}_{-}(\mathbf{x},0)=\tilde{f}_{+}(\mathbf{x},0)=0 \quad\text{for a.a. } \mathbf{x}\in \Omega.
	\end{equation}	
	Our system will be supplemented by boundary conditions
	\begin{equation}\label{Eq:zero boundary conditions}
	\begin{aligned}
	& u=v=0\quad \text{on }\Gamma_{D},\\
	& \frac{\partial u}{\partial n} = \frac{\partial v}{\partial n} = 0 \quad \text{on}\ \Gamma_{N},
	\end{aligned}	
	\end{equation}
	where $n$ is the unit outward-pointing normal vector of the boundary $\partial\Omega$ and $\Gamma_{N},\Gamma_{D}$ are open disjoint subsets of $\partial\Omega$, $\partial\Omega = \overline{\Gamma_{D}}\cup \overline{\Gamma_{N}}$.\\
	\indent Apparently the problem $(\ref{Eq:general model}),(\ref{Eq:zero boundary conditions})$ has always the trivial solution $[0,0]$. Our system should describe a reaction of two chemicals, e.g. morphogens, having a basic positive spatially constant steady state $[\overline{u},\overline{v}]$, that means we should assume in fact $f(\overline{u},\overline{v})=g(\overline{v},\overline{v})=\tilde{f}_{-}(\mathbf{x},\overline{u})=\tilde{f}_{+}(\mathbf{x},\overline{u})=0$ instead of (\ref{Eq:zero fixed point assumption}), but as usually, we can shift the positive steady state to zero and we obtain our system satisfying (\ref{Eq:zero fixed point assumption}). Let us emphasize that then the functions $u,v$ do not describe concentrations of the reactants, but their differences from the basic constant stationary state 
	$[\overline{u},\overline{v}]$.\\	
	\indent We will consider assumptions under which the problem $(\ref{Eq:general model}),(\ref{Eq:zero boundary conditions})$ with
	$\tilde{f}_{-}\equiv \tilde{f}_{+} \equiv 0$ exhibits diffusion driven instability discovered in the famous Turing's paper \cite{Turing}. That means if $\tilde{f}_{-}\equiv \tilde{f}_{+} \equiv 0$ then the trivial solution $[0,0]$ is stable as a solution of the corresponding problem without diffusion (ODE's obtained for $d_1=d_2=0$), but as a solution  of the whole system it is unstable for $[d_1,d_2]$ from a certain subdomain $D_U$ of the positive quadrant $\mathbb{R}_{+}^{2}$ (domain of instability), and stable only for $[d_1,d_2] \in D_S=\mathbb{R}_{+}^{2}\setminus \overline{D_U}$ (domain of stability). Spatially non-homogeneous steady states bifurcate from the basic constant equilibrium in some points of $\overline{D_U}$, but such a bifurcation is excluded in $D_S$. Let us note that spatially non-homogeneous steady states can describe spatial patterns in some models in biology.\\
	\indent {\bf Our goal} is to prove that if we add unilateral terms $\tilde{f}_{-}(\mathbf{x},u^-)$, $\tilde{f}_{+}(\mathbf{x},u^+)$, then the domain of diffusion coefficients where spatially non-homogeneous steady states can bifurcate is smaller than $\overline{D_U}$. In fact we will prove more, see below. An example of unilateral terms can be
	\begin{equation*}
	\tilde{f}_{-}(\mathbf{x},u^-)=s_{-}(\mathbf{x})\frac{u^-}{1+\varepsilon u^-},\quad \tilde{f}_{+}(\mathbf{x},u^+)=s_{+}(\mathbf{x})\frac{u^+}{1+\varepsilon u^+}.
	\end{equation*}	
	The stationary system corresponding to (\ref{Eq:general model}) can be written in the form
	\begin{equation}\label{Eq:stationary problem nonlinear with unilateral terms}
	\begin{aligned}
	d_{1}\Delta u + b_{1,1}u + b_{1,2}v + n_{1}(u,v) + 
	\tilde{f}_{-}(\mathbf{x},u^-) - \tilde{f}_{+}(\mathbf{x},u^+) &=0,\\
	d_{2}\Delta v + b_{2,1}u + b_{2,2}v + n_{2}(u,v) &=0, 
	\end{aligned}		
	\end{equation}
	where $\mathbf{B}:=(b_{i,j})_{i,j=1,2}$ is the Jacobi matrix of the mappings $f,g$ at $[0,0]$ and 
	the functions $n_{1},n_{2}$ are higher order terms, i.e.    
	\begin{equation}\label{Eq:nonlinearities small at zero}
	n_{1,2}(u,v)= o(|u|+|v|) \text{ as } |u|+|v|\rightarrow 0. 
	\end{equation}
	(The nonlinear part in the first equation could be written also in the form $s_{-}(\mathbf{x}) u^{-} - s_{+}(\mathbf{x})u^{+} + \tilde n_{1}(\mathbf{x},u,v)$, that means a homogenization + higher order terms dependent on $\mathbf{x}$).\\
	\indent We will always assume that the following conditions necessary for Turing's diffusion driven instability mentioned above are fulfilled:
	\begin{equation}\label{Eq:B assumption}
	b_{1,1}>0, b_{2,2}<0, b_{1,2}b_{2,1}<0, tr(\mathbf{B})<0, det(\mathbf{B})>0.
	\end{equation}
	The first three conditions in $(\ref{Eq:B assumption})$ correspond to an activator-inhibitor system (for $b_{1,2}<0, b_{2,1}>0$),  or to a substrate depletion system 
	(for $b_{1,2}>0, b_{2,1}<0$), see e.g. \cite{Murray2}. The last two conditions ensure the stability of $[0,0]$ as a solution of the system without any diffusion.\\	
	\indent We will work mainly with the homogenized system 
	\begin{equation}\label{Eq:stationary problem with unilateral terms}
	\begin{aligned}
	d_{1}\Delta u + b_{1,1}u + b_{1,2}v + 
	s_{-}(\mathbf{x})u^{-} - s_{+}(\mathbf{x})u^{+} &= 0,\\
	d_{2}\Delta v + b_{2,1}u + b_{2,2}v &= 0.
	\end{aligned}
	\end{equation}
	We will show more than what is mentioned above, namely that critical points, i.e. couples $[d_1,d_2]$ for which the homogenized problem	(\ref{Eq:stationary problem with unilateral terms}), (\ref{Eq:zero boundary conditions}) has a non-trivial solution, can exist only in a smaller domain than in the classical case $\tilde{s}_{-}=\tilde{s}_{+}\equiv 0$. Since any bifurcation point is simultaneously a critical point, the main goal mentioned above will follow. A similar result was proved in \cite{Kucera-u} for the case of unilateral sources on the boundary described by quasi-variational inequalities, but we consider the description of unilateral sources and sinks by the terms $\tilde{f}_{-}(\mathbf{x},u^-)$, $\tilde{f}_{+}(\mathbf{x},u^+)$ more natural. We will briefly discuss also problems with unilateral terms of the type $s_-(x)u^-$, $s_+(x)u^+$ on the boundary.\\	
	\indent {\bf Main ideas} are similar to those from \cite{Kucera-u}. Considering a weak formulation, we will write our problem as a system of operator equations in Sobolev space and we will consider an arbitrary fixed $d_2$. Expressing the variable $v$ from the second equation and substituting it to the first equation, we reduce the originally non-symmetric problem to a single equation with a positively homogeneous operator having a potential. A variational characterization of its largest eigenvalue enables us to compare the largest eigenvalue corresponding to the problem with and without unilateral terms, which is simultaneously the largest $d_1$ for which $[d_1,d_2]$ is a critical point of the original system with and without unilateral terms.\\	
	\indent Let us note that if unilateral sources of the second variable $v$ (inhibitor) are supplemented in the second equation then bifurcation of spatial patterns occurs even in the domain $D_S$, where it is excluded for the classical case without unilateral sources. See e.g. \cite{Kucera-Vath} and references therein for the case of sources described by variational inequalities, \cite{Eisner-Kucera-Vath} for unilateral sources described by multivalued maps and \cite{Eisner-Kucera}, \cite{Kucera-Navratil-Dirichlet} for the case of unilateral terms similar to the current paper. These results motivated numerical experiments \cite{Vejchodsky-Kucera} showing that for a concrete model also spatial patterns arise from small initial perturbations for diffusion parameters from $D_S$, where it is not the case without unilateral sources. The sense of these results is positive because one of the problems of Turing's theory is that the set of diffusion parameters for which diffusion-driven instability occurs is too small, so unilateral sources for $v$ improve this situation. The result of the current paper is opposite, unilateral sources for $u$ makes larger the set of diffusion parameters for which bifurcation of spatial patterns is exluded, i.e. for which no small spatial patterns can exist. We believe that, at least in some cases, it is a signal that the same is true for the set of parameters for which spatial patterns evolve from small perturbations of the basic spatially constant steady state. It agrees with numerical experiments which will be published in a forthcoming paper. This seems to be a negative result, but perhaps there are situations when it would be valuable to understand how to prevent evolution of spatial patterns. For instance, patterns play a role in models of tumors, see e.g. \cite{tumors} and references therein. In spite of that the paper \cite{tumors} has completely different goals, it can be perhaps motivating from the point of view mentioned, in particular its Section 5.\\	
	\indent We present the basic general assumptions and definitions in Section $\ref{Sec:basic ass a def}$. Main results of this paper are formulated and discussed in Section \ref{Sec:main results}. In Section \ref{Sec:abstract fomulation} we formulate our problem as a system of operator equations in Sobolev space and we describe properties of the corresponding operators. Section \ref{Sec:fixed d2} concerns a reduction of our system to a single equation with a positively homogeneous operator and a variational characterization of its largest eigenvalue.	A comparison of largest eigenvalues and consequently also critical points with and without unilateral terms by using this variational characterization is given. The proofs of the main results are done in Section \ref{Sec:proofs}.
	
	\section{Basic assumptions and definitions}\label{Sec:basic ass a def}
	We will always suppose that there exists $c\in\mathbb{R}$ such that
	\begin{eqnarray}
	|n_{j}(\chi,\xi)|\leq c(1+|\chi|^{q-1} + |\xi|^{q-1})\quad\text{for all } \chi,\xi\in\mathbb{R}, j=1,2,\label{Eq:nonlinearities constrains}\\
	|\tilde{f}_{\mp}(\mathbf{x},\xi)|\leq c(1+|\xi|^{q-1})\quad\text{for all } \xi\in\mathbb{R} \text{ and a.a. } \mathbf{x}\in\Omega,\label{Eq:tilda f constrains}
	\end{eqnarray} 
	with some $q>2$ if $N=2$ or $2<q<\frac{2N}{N-2}$ if $N>2$. In the dimension $N=1$ no growth assumptions are necessary.\\
	\indent Besides systems $(\ref{Eq:stationary problem with unilateral terms})$ and $(\ref{Eq:stationary problem nonlinear with unilateral terms})$ we will discuss systems
	\begin{equation}\label{Eq:stationary problem}
	\begin{aligned}
	d_{1}\Delta u + b_{1,1}u + b_{1,2}v = 0,\\
	d_{2}\Delta v + b_{2,1}u + b_{2,2}v = 0
	\end{aligned}
	\end{equation}
	and
	\begin{equation}\label{Eq:stationary problem nonlinear}
	\begin{aligned}
	d_{1}\Delta u + b_{1,1}u + b_{1,2}v +n_{1}(u,v) = 0,\\
	d_{2}\Delta v + b_{2,1}u + b_{2,2}v +n_{2}(u,v) = 0.
	\end{aligned}
	\end{equation}
	\indent By solutions we will always mean weak solutions in the space	
	\begin{equation}\label{Eq:define function space}
	H^{1}_{D}(\Omega) := \{\phi\in W^{1,2}(\Omega):\, \phi=0 \text{ on } \Gamma_{D} \text{ in the sense of traces}\}.
	\end{equation}
	If $\Gamma_{D} = \emptyset$, then the space $H^{1}_{D}$ is actually the whole Sobolev space $W^{1,2}$ equipped with the standard inner product
	\begin{equation}\label{Eq:inner product N}
	(u,\varphi)_{H^{1}_{D}} = (u,\varphi)_{W^{1,2}} = \int_{\Omega} \left(\nabla u \nabla\varphi + u \varphi\right) \ d\Omega
	\end{equation}
	and the Sobolev norm $\lVert u \rVert_{W^{1,2}} = \left(\int_{\Omega} (\nabla u)^{2} + u^{2} \ d\Omega\right)^{\frac{1}{2}}$. If $\Gamma_{D}\neq\emptyset$, then we will use the inner product
	\begin{equation}\label{Eq:inner product D-N}
	(u,\varphi)_{H^{1}_{D}} = \int_{\Omega} \nabla u \nabla\varphi \ d\Omega
	\end{equation}
	and the norm $\lVert u\rVert_{H^{1}_{D}} = \left( \int_{\Omega} (\nabla u)^2 \ d\Omega\right)^{\frac{1}{2}}$ equivalent to the classical Sobolev norm. 
	\begin{definition}[Critical point]\label{Def:critical point}\mbox{}\\
		A parameter $d=[d_{1},d_{2}]\in\mathbb{R}^{2}_{+}$ will be called a critical point of $(\ref{Eq:stationary problem}),(\ref{Eq:zero boundary conditions})$ or $(\ref{Eq:stationary problem with unilateral terms}),(\ref{Eq:zero boundary conditions})$ if there exists a non-trivial (weak) solution of $(\ref{Eq:stationary problem}),(\ref{Eq:zero boundary conditions})$ or $(\ref{Eq:stationary problem with unilateral terms}),(\ref{Eq:zero boundary conditions})$, respectively.
	\end{definition}
	\begin{definition}[Bifurcation point]\label{Def:bifurcation point}\mbox{}\\
		A parameter $d^{0}=[d^{0}_{1},d^{0}_{2}]\in\mathbb{R}^{2}_{+}$ will be called a bifurcation point of $(\ref{Eq:stationary problem nonlinear}),(\ref{Eq:zero boundary conditions})$ or $(\ref{Eq:stationary problem nonlinear with unilateral terms}),(\ref{Eq:zero boundary conditions})$ if in any neighbourhood of $[d^{0},0,0]\in\mathbb{R}^{2}_{+}\times H^{1}_{D} \times H^{1}_{D}$ there exists $[d,W] = [d,u,v]$, $\lVert W \rVert\neq 0$ satisfying  $(\ref{Eq:stationary problem nonlinear}),(\ref{Eq:zero boundary conditions})$ or $(\ref{Eq:stationary problem nonlinear with unilateral terms}),(\ref{Eq:zero boundary conditions})$, respectively.
	\end{definition}
	\begin{remark}\label{Rem:Laplacian eigenvalues}
		Let's consider the problem
		\begin{equation}\label{Eq:Laplace eigenvalue problem}
		\begin{aligned}
		-\Delta u &= \kappa u,\\
		u&=0\text{ on }\Gamma_{D},\\
		\frac{\partial u}{\partial n} &= 0\text{ on }\Gamma_{N}.
		\end{aligned}
		\end{equation}
		The eigenvalues of $(\ref{Eq:Laplace eigenvalue problem})$ form a non-negative non-decreasing sequence $\kappa_{j}$ with $j = 1,2,\ldots$ (for $\Gamma_{D}\neq\emptyset$) or $j=0,1,2,\ldots$ (for $\Gamma_{D}=\emptyset$). The first eigenvalue is always simple. In the case $\Gamma_{D}\neq\emptyset$, the eigenfunction $e_{1}$ corresponding to the first eigenvalue $\kappa_{1}$ does not change the sign on the domain $\Omega$. In the case $\Gamma_{D}=\emptyset$, the eigenfunction $e_{0}$ corresponding to the first eigenvalue $\kappa_{0}=0$ is constant. Other eigenfunctions change the sign in both cases. We can choose an orthonormal basis $e_{j}$ in $H^{1}_{D}$, $j = 1,2,\ldots$ (for $\Gamma_{D}\neq\emptyset$) or $j=0,1,2,\ldots$ (for $\Gamma_{D}=\emptyset$) composed of the eigenfunctions of $(\ref{Eq:Laplace eigenvalue problem})$.
	\end{remark}
	\indent Let's remind that the conditions $(\ref{Eq:B assumption})$ are always considered. The sets
	\begin{equation}\label{Eq:hyperbolas}
	C_{j}:=\left\{[d_{1},d_{2}]\in\mathbb{R}^{2}_{+}:\ d_{1} = \frac{1}{\kappa_{j}}\left(\frac{b_{1,2}b_{2,1}}{d_{2}\kappa_{j}-b_{2,2}}+b_{1,1}\right) \right\},\; j=1,2,\ldots
	\end{equation}
	are hyperbolas (or more specifically their parts) in the positive quadrant $\mathbb{R}^{2}_{+}$. Let's note that we present hyperbolas in the different form than usually, namely with respect to $d_{1}$. It is of course equivalent to the standard form derived from the relation
	\begin{equation*}
	(\kappa_{j}d_{1}-b_{1,1})(\kappa_{j}d_{2}-b_{2,2})-b_{1,2}b_{2,1}=0
	\end{equation*}
	(see e.g. \cite{Murray2}). If $\Gamma_{D}=\emptyset$, for $j=0$ the last equality is never satisfied, because $det(B)$ is positive by $(\ref{Eq:B assumption})$. 
	The envelope
	\begin{equation}\label{Eq:envelope}
	C_{E} := \left\{d=[d_{1},d_{2}]\in\mathbb{R}^{2}_{+}:\ d_{1}=\max_{\tilde{d_{1}}\in\mathbb{R}_{+}}\left\{\tilde{d_{1}}: [\tilde{d_{1}},d_{2}]\in \bigcup^{\infty}_{j=1} C_{j}\right\} \right\}
	\end{equation}
	divides the positive quadrant $\mathbb{R}^{2}_{+}$ onto two sets $D_{U}$ and $D_{S}$ (see Figure $\ref{Fig:hyperbolas}$).
	\begin{figure}[h!]
		\centering
		\includegraphics[scale=0.32]{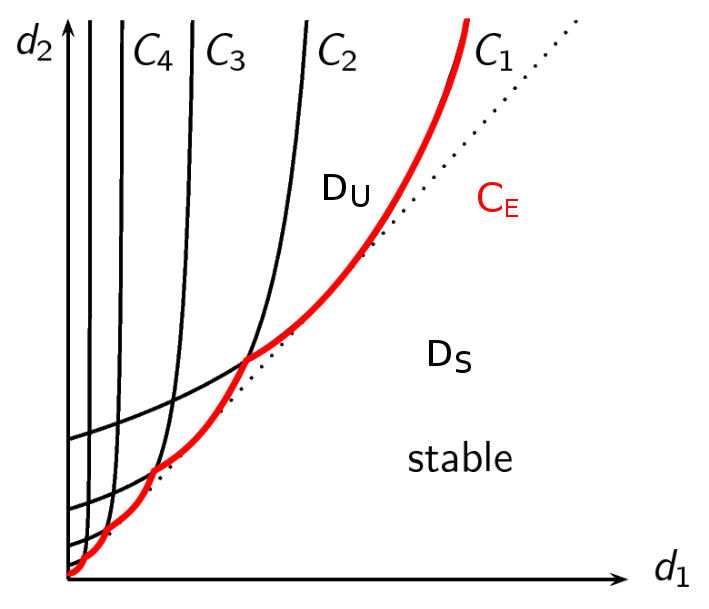}
		\caption{Illustration of the hyperbolas $C_{j}$ and the envelope $C_{E}$. The case when all eigenvalues $\kappa_{j}$ are simple.}\label{Fig:hyperbolas}
	\end{figure}
	\begin{remark}\label{Rem:conciding hyperbolas}
		If all eigenvalues of $(\ref{Eq:Laplace eigenvalue problem})$ are simple, i.e. $\kappa_{j}<\kappa_{j+1}$ for all $j\in\mathbb{N}$, then $C_{j}\neq C_{j+1}$ for all $j>0$. If an eigenvalue $\kappa_{j}$ has a multiplicity $k$, then $\kappa_{j-1}<\kappa_{j}=\ldots=\kappa_{j+k-1}<\kappa_{j+k}$ and $C_{j-1}\neq C_{j}=\ldots =C_{j+k-1}\neq C_{j+k}$. The sets
		\begin{equation*}
		\begin{aligned}
		D_{U} := \{d=[d_{1},d_{2}]\in\mathbb{R}^{2}_{+}: d\text{ is on the left of }C_{E} \},\\
		D_{S} := \{d=[d_{1},d_{2}]\in\mathbb{R}^{2}_{+}: d\text{ is on the right of }C_{E} \}
		\end{aligned}
		\end{equation*}
		are called the domain of instability and the domain of stability. It is known that if $[d_{1},d_{2}]\in D_{S}$, then all eigenvalues $\lambda$ of the problem deciding about stability of the trivial solution of the evolution system corresponding to $(\ref{Eq:stationary problem nonlinear}),(\ref{Eq:zero boundary conditions})$ have negative real parts and if $[d_{1},d_{2}]\in D_{U}$, then there is an eigenvalue $\lambda$ with positive real part (for a particular case see \cite{Mimura},\cite{Nishiura} and for a general case \cite{Eisner-Kucera-spatial}). In particular, the trivial solution of $(\ref{Eq:stationary problem nonlinear}),(\ref{Eq:zero boundary conditions})$ is linearly stable for  $[d_{1},d_{2}]\in D_{S}$ and unstable for  $[d_{1},d_{2}]\in D_{U}$.
	\end{remark}	
	\begin{remark}\label{Rem:crit point <=> lies on C_j}
	The following properties of the curves $C_j$ are known, see e.g. \cite{Nishiura},\cite{Mimura} for a particular case, or \cite{Eisner-Kucera-spatial} for the general case.
		\begin{itemize}
			\item{A point $d=[d_{1},d_{2}]$ is a critical point of $(\ref{Eq:stationary problem}),(\ref{Eq:zero boundary conditions})$ if and only if there exists  $j$ such that $d\in C_{j}$.} In particular, the domain of stability $D_{S}$ does not contain any critical point of $(\ref{Eq:stationary problem}),(\ref{Eq:zero boundary conditions})$ or bifurcation point of $(\ref{Eq:stationary problem nonlinear}),(\ref{Eq:zero boundary conditions})$. Under some additional assumptions, e.g. if the eigenvalue $\kappa_j$ is simple or of odd multiplicity, the points on $C_j$ are simultaneously bifurcation points (see e.g. \cite{Nishiura}).  
			\item{If $d\in C_{n}$ for $n=j, \ldots, j+k-1$ (either $k$ is the multiplicity of the eigenvalue $\kappa_{j}$ or $d$ is in the intersection of two hyperbolas $C_{j}$,$C_{m}$ and $k$ is the sum of multiplicities of $\kappa_{j},\kappa_{m}$, see Remark \ref{Rem:conciding hyperbolas}), then $\text{span}\left(\left[\frac{d_{2}\kappa_{j}-b_{2,2}}{b_{2,1}}e_{j},e_{j}\right]^{j+k-1}_{n=j}\right)$ is the set of the solutions of $(\ref{Eq:stationary problem}),(\ref{Eq:zero boundary conditions})$.}
			\end{itemize} 
	\end{remark}
	
	\section{Main results}\label{Sec:main results}
	Let's recall that the assumptions $(\ref{Eq:nonlinearities constrains}),(\ref{Eq:tilda f constrains})$ are automatically supposed. Besides the notions introduced in Section \ref{Sec:basic ass a def} we will use the following symbols.
	\begin{notation}\label{Def:def compact sets CRr}\mbox{}\\
		Let $r,R,\varepsilon\in\mathbb{R}_{+}$ and $r < R$. We define\newline
		$C^{R}_{r} := \{d=[d_{1},d_{2}]\in C_{E}: d_{2}\in [r,R] \}$,\newline
		$C^{R}_{r}(\varepsilon) := \{d=[d_{1},d_{2}]\in C_{E}\cup D_{U}: d_{2}\in [r,R] \wedge dist(d,C_{E})<\varepsilon \}$.
	\end{notation}
	The following theorem is the main result of this paper.	
	\begin{theorem}\label{Thm:General case}
		\begin{enumerate}[i)]
			\item{The domain of stability $D_{S}$ contains neither critical points of $(\ref{Eq:stationary problem with unilateral terms}),(\ref{Eq:zero boundary conditions})$ nor bifurcation points of $(\ref{Eq:stationary problem nonlinear with unilateral terms}),(\ref{Eq:zero boundary conditions})$.}
			\item{Let $0 < r < R$. Let $C_{j},\ldots,C_{j+k-1}$ be all hyperbolas which have a non-empty intersection with $C^{R}_{r}$. Let any linear combination $e$ of the eigenfunctions of $(\ref{Eq:Laplace eigenvalue problem})$ corresponding to $\kappa_{j},\ldots,\kappa_{j+k-1}$ satisfy
			\begin{equation}\label{Eq:main thm cond}
				s_{-}e^{-}-s_{+}e^{+}\not\equiv 0.
			\end{equation}
			Then there exists $\varepsilon>0$ such that there are neither critical points of $(\ref{Eq:stationary problem with unilateral terms}),(\ref{Eq:zero boundary conditions})$ nor bifurcation points of $(\ref{Eq:stationary problem nonlinear with unilateral terms}),(\ref{Eq:zero boundary conditions})$ in $C^{R}_{r}(\varepsilon)$.}
		\end{enumerate}
	\end{theorem}
	We emphasize that if the condition $(\ref{Eq:main thm cond})$ is not satisfied for some linear combination $e$ mentioned, then there are critical points of $(\ref{Eq:stationary problem with unilateral terms}), (\ref{Eq:zero boundary conditions})$ directly on $C^{R}_{r}$ due to Remark \ref{Rem:crit point <=> lies on C_j}. Let's note that if all hyperbolas $C_{j},\ldots,C_{j+k-1}$ do not coincide, i.e. it is not $\kappa_{j}=\kappa_{j+1}=\ldots=\kappa_{j+k-1}$, then the eigenfunctions $e_{j},\ldots,e_{j+k-1}$ do not correspond to the same eigenvalue and their linear combination need not be an eigenfunction. We discuss possible situations in the following two examples:
	\begin{itemize}
		\item{First let's assume that $C^{R}_{r}$ has a non-empty intersection with exactly two non-coinciding hyperbolas $C_{k}$ and $C_{k+1}$. If both $e=e_{k}$ and $e=e_{k+1}$ satisfy $(\ref{Eq:main thm cond})$, then there are no critical points of $(\ref{Eq:stationary problem with unilateral terms}),(\ref{Eq:zero boundary conditions})$ on 
		$C^{R}_{r}\setminus (C_{k}\cap C_{k+1})$. However, it can happen that there is a linear combination $e$ of $e_{k}$, $e_{k+1}$ such that $s_{-}e^{-}-s_{+}e^{+}\equiv 0$, and in this case the intersection point $C_{k}\cap C_{k+1}$ is a critical point of $(\ref{Eq:stationary problem with unilateral terms}), (\ref{Eq:zero boundary conditions})$ (see also Remark \ref{Rem:crit point <=> lies on C_j}).}
		\item{In an other scenario we take $C^{R}_{r}$ which consists of a part of two coinciding hyperbolas $C_{k}=C_{k+1}$, i.e. $\kappa_{k}=\kappa_{k+1}$. In this case the assumption of Theorem \ref{Thm:General case} ii) means that every eigenfunction corresponding to $\kappa_{k}=\kappa_{k+1}$ must satisfy $(\ref{Eq:main thm cond})$. Otherwise the critical points of $(\ref{Eq:stationary problem with unilateral terms}), (\ref{Eq:zero boundary conditions})$ are on the whole $C_{k}$, in particular on $C^{R}_{r}$ (see Remark \ref{Rem:crit point <=> lies on C_j}).}
	\end{itemize}	
	The result is illustrated on Figure $\ref{Fig:analytic result general}$.\\	
	\begin{figure}[H]
		\centering
		\includegraphics[scale=0.32]{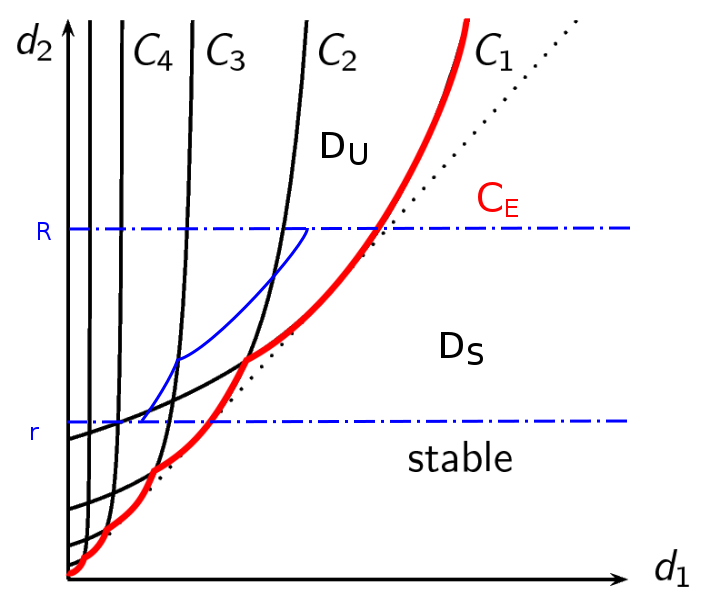}
		\caption{Illustration of the result of Theorem $\ref{Thm:General case}$. The critical points are no longer in the region between $C_E$ (red curve) and blue curve. Assuming the case when all eigenvalues $\kappa_{j}$ are simple, i.e. $C_{j}\neq C_{k}$ for all $k\neq j$, and any linear combination of eigenfunctions $e_{1},e_{2}$ corresponding to $\kappa_{1},\kappa_{2}$ satisfy $(\ref{Eq:main thm cond})$.}\label{Fig:analytic result general}
	\end{figure}
	\begin{corollary}\label{Corol}
		\begin{enumerate}[i)]
			\item{For any compact part $M$ of $D_{S}$ there exists $\delta>0$ such that for any $[d_{1},d_{2}]\in M$ there are no non-trivial solutions of $(\ref{Eq:stationary problem nonlinear with unilateral terms}),(\ref{Eq:zero boundary conditions})$ with $0< \lVert u \rVert_{H^{1}_{D}} + \lVert v \rVert_{H^{1}_{D}} < \delta$.}
			\item{Under the assumption from Theorem \ref{Thm:General case} ii), for any compact part $M$ of $D_{S}\cup C^{R}_{r}(\varepsilon)$ there exists $\delta>0$ such that for any $[d_{1},d_{2}]\in M$ there are no non-trivial solutions of $(\ref{Eq:stationary problem nonlinear with unilateral terms}),(\ref{Eq:zero boundary conditions})$ with $0 < \lVert u \rVert_{H^{1}_{D}} + \lVert v \rVert_{H^{1}_{D}} < \delta$.}
		\end{enumerate}		 
	\end{corollary}
	\begin{proof}\mbox{}\\
		Indeed, it is easy to see that if this were not true, then a bifurcation point of $(\ref{Eq:stationary problem nonlinear with unilateral terms}),(\ref{Eq:zero boundary conditions})$ would exist in $M$, which would contradict Theorem \ref{Thm:General case}. 
	\end{proof}
	\noindent There are two important particular cases for $\Gamma_{D}\neq\emptyset$ and $\Gamma_{D}=\emptyset$:
	\begin{theorem}\label{Thm:Special case Dirichlet}
		Let $\Gamma_{D}\neq\emptyset$. Let one of the functions $s_{+},s_{-}$ be identically zero and the other positive a.e. on $\Omega$. Let $d^{I}_{2}$ be the second coordinate of the intersection point of $C_{1}$ and $C_{2}$. 
		\begin{enumerate}[i)]
			\item{Any $d\in C_{1}$, in particular any $d\in C^{R}_{r}$ with $d^{I}_{2}\leq r < R$, is a critical point of $(\ref{Eq:stationary problem with unilateral terms})$,$(\ref{Eq:zero boundary conditions})$.}
			\item{If $0< r < R < d^{I}_{2}$, then there exists $\varepsilon>0$ such that there are neither critical points of $(\ref{Eq:stationary problem with unilateral terms})$,$(\ref{Eq:zero boundary conditions})$ nor bifurcation points of $(\ref{Eq:stationary problem nonlinear with unilateral terms}),(\ref{Eq:zero boundary conditions})$ in $C^{R}_{r}(\varepsilon)$.}
		\end{enumerate}		
	\end{theorem}
	\begin{theorem}\label{Thm:Special case Neumann}
		Let $\Gamma_{D}=\emptyset$. Let one of the functions $s_{+},s_{-}$ be identically zero and the other positive a.e. on $\Omega$. Then for any $0 < r < R$ there exists $\varepsilon>0$ such that there are neither critical points of $(\ref{Eq:stationary problem with unilateral terms})$,$(\ref{Eq:zero boundary conditions})$ nor bifurcation points of $(\ref{Eq:stationary problem nonlinear with unilateral terms}),(\ref{Eq:zero boundary conditions})$ in $C^{R}_{r}(\varepsilon)$.
	\end{theorem}
	\begin{remark}	
		The size of $\varepsilon$ in Theorems $\ref{Thm:General case}$-$\ref{Thm:Special case Neumann}$ depends on $r$ and $R$. Actually $\varepsilon\rightarrow 0$ as $R\rightarrow d^{I}_{2}$ or $r\rightarrow 0$ in Theorem $\ref{Thm:Special case Dirichlet}$ and $\varepsilon\rightarrow 0$ as $r\rightarrow 0$ in Theorem $\ref{Thm:Special case Neumann}$. The following theorem states that if the source and sink are in some sense small enough, then there exists at least one critical point $[d_1,d_2] \in D_{U}\cup C_{E}$ with a given $d_2$. A question if sometimes (for a strong source or sink) no critical point with a given $d_2$ exists remains an open problem. Cf. Remark \ref{Rem:existence of maxima} in Section \ref{Sec:fixed d2}.
	\end{remark}
	\begin{theorem}\label{Thm:tau estimate consequence}
		Let $d_{2}>0$ be arbitrary fixed. Let $j_{0}$ be such that $\left[\frac{1}{\kappa_{j_{0}}}\left(\frac{b_{1,2}b_{2,1}}{d_{2}\kappa_{j_{0}}-b_{2,2}}+b_{1,1}\right),d_{2}\right]\in C_{E}$ (see $(\ref{Eq:hyperbolas})$,$(\ref{Eq:envelope})$). If $\max\left\{\lVert s_{-}\rVert_{\infty},\lVert s_{+}\rVert_{\infty}\right\} < b_{1,1} + \frac{b_{1,2}b_{2,1}}{d_{2}\kappa_{j_{0}}-b_{2,2}}$, then there exists at least one $d_{1}$ such that $[d_{1},d_{2}]\in D_{U}\cup C_{E}$ is a critical point of the problem $(\ref{Eq:stationary problem with unilateral terms}),(\ref{Eq:zero boundary conditions})$.
	\end{theorem}
	The last theorem of this section is a modification of Theorem $\ref{Thm:General case}$ for the case of unilateral terms in boundary conditions, namely for systems $(\ref{Eq:stationary problem})$ and $(\ref{Eq:stationary problem nonlinear})$ with boundary conditions
	\begin{equation}\label{Eq:boundery conditions with unilateral terms}
	\begin{aligned}
	u = v = 0\quad\text{on}\quad \Gamma_{D},\\
	\frac{\partial u}{\partial n} = s_{-}(\mathbf{x})u^{-} - s_{+}(\mathbf{x})u^{+}\quad\text{on}\quad\Gamma_{N},\\
	\frac{\partial v}{\partial n} = 0 \quad\text{on}\quad \Gamma_{N}.
	\end{aligned}
	\end{equation}	
		Let us note that we consider only positively homogeneous boundary conditions because introducing more general boundary terms as $\tilde{f}_{\mp}$ in the case of sources and sinks in the interior of the domain would mean additional technical complications.
	\begin{theorem}\label{Thm:General case source on boundary}
		\begin{enumerate}[i)]
			\item{The domain of stability $D_{S}$ contains neither critical points of $(\ref{Eq:stationary problem}),(\ref{Eq:boundery conditions with unilateral terms})$ nor bifurcation points of $(\ref{Eq:stationary problem nonlinear}),(\ref{Eq:boundery conditions with unilateral terms})$.}
			\item{Let $0 < r < R$. Let $C_{j},\ldots,C_{j+k-1}$ be all hyperbolas which have a non-empty intersection with $C^{R}_{r}$. Let any linear combination $e$ of the eigenfunctions of $(\ref{Eq:Laplace eigenvalue problem})$ corresponding to $\kappa_{j},\ldots,\kappa_{j+k-1}$ satisfy
			\begin{equation}\label{Eq:main thm cond on boundary}
				s_{-}e^{-}-s_{+}e^{+}\not\equiv 0\quad \text{ on }\Gamma_{N}.
			\end{equation}
			Then there exists $\varepsilon>0$ such that there are neither critical points of $(\ref{Eq:stationary problem}),(\ref{Eq:boundery conditions with unilateral terms})$ nor bifurcation points of $(\ref{Eq:stationary problem nonlinear}),(\ref{Eq:boundery conditions with unilateral terms})$ in $C^{R}_{r}(\varepsilon)$.}
		\end{enumerate}
	\end{theorem}
	Analogous consequence as in Corollary $\ref{Corol}$ can be formulated for Theorems \ref{Thm:Special case Dirichlet}, \ref{Thm:Special case Neumann} and \ref{Thm:General case source on boundary}.
	
	\section{Abstract formulation}\label{Sec:abstract fomulation}
	\indent We define the operator $A: H^{1}_{D}\mapsto H^{1}_{D}$ as
	\begin{equation}\label{Eq:define A}
	(Au,\varphi) = \int_{\Omega} u \varphi\ d\Omega\quad \text{ for all } u,\varphi\in H^{1}_{D}(\Omega).
	\end{equation}
	\begin{remark}\label{Rem:eigenvalues of A}
		The operator $A$ defined by $(\ref{Eq:define A})$ is linear, bounded, symmetric and compact due to compact embedding $W^{1,2}\hookrightarrow\hookrightarrow L^{2}$. Simple calculation gives that the eigenvalues of the operator $A$ are $\mu_{j} = \frac{1}{\kappa_{j}}, j=1,2,\ldots$ for $\Gamma_{D}\neq\emptyset$ and $\mu_{j} = \frac{1}{\kappa_{j}+1},j=0,1,2,\ldots$ for $\Gamma_{D}=\emptyset$, and the corresponding eigenvectors of $A$ coincide with the eigenfunctions $e_j$ of $(\ref{Eq:Laplace eigenvalue problem})$. In particular, the maximal eigenvalue of $A$ is always one and therefore $(Au,u)\leq\lVert u \rVert^{2}_{H^{1}_{D}}$, where  the equality holds only for all multiples $u$ of $e_{1}$ or $e_{0}$ if $\Gamma_{D}\neq\emptyset$ or $\Gamma_{D}=\emptyset$, respectively, see also Remark \ref{Rem:Laplacian eigenvalues}. Hence, $((I-A)u,u)>0$ for all $u\notin span\{e_1\}$ in the case $\Gamma_{D}\neq\emptyset$ and for all $u\notin span\{e_0\}$ in the case $\Gamma_{D}=\emptyset$.
	\end{remark}	
	\indent We define two non-linear operators $N_{1},N_{2}: H^{1}_{D}\times H^{1}_{D}\mapsto H^{1}_{D}$ as
	\begin{equation}\label{Eq:definition of nonlinear operators}
	(N_{i}(u,v),\varphi) = \int_{\Omega} n_{i}(u,v)\varphi\ d\Omega\quad \text{ for all } u,v,\varphi\in H^{1}_{D}, \ i=1,2.
	\end{equation}
	These two operators are well-defined and continuous due to the theorem about Nemytskii operators and the assumptions $(\ref{Eq:nonlinearities constrains})$.
	\begin{remark}\label{Rem:nonlin op limit cond}
		It is known that under the assumptions $(\ref{Eq:nonlinearities small at zero})$,$(\ref{Eq:nonlinearities constrains})$ we have
		\begin{equation}\label{Eq:nonlinear operator limit condition}
		\lim\limits_{\lVert u\rVert_{H^{1}_{D}} + \lVert v\rVert_{H^{1}_{D}}\rightarrow 0} \frac{N_{i}(u,v)}{\lVert u\rVert_{H^{1}_{D}} + \lVert v\rVert_{H^{1}_{D}}} = 0,\quad i=1,2.
		\end{equation}
		For details see e.g. Appendix A.1 of \cite{Kucera-u}.
	\end{remark}
	Furthermore we define operators $\beta^{-},\beta^{+}: H^{1}_{D}\mapsto H^{1}_{D}$ by
	\begin{equation}\label{Eq:define betas}
	\begin{aligned}
	(\beta^{\mp}(u),\varphi) &= \mp\int_{\Omega} s_{\mp}u^{\mp} \varphi\ d\Omega \quad\text{ for all } u,\varphi\in H^{1}_{D}
	\end{aligned}	
	\end{equation}
	and $\beta: H^{1}_{D}\mapsto H^{1}_{D}$ as
	\begin{equation}\label{Eq:define beta}
	\beta:= \beta^{+}+\beta^{-}.	
	\end{equation}
	Due to the theorem about Nemytskii operators and $(\ref{Eq:tilda f constrains})$ we can also define operators $\tilde{F}_{-},\tilde{F}_{+}: H^{1}_{D}\mapsto H^{1}_{D}$ by
	\begin{equation}\label{Eq:define F operators}
	\begin{aligned}
	(\tilde{F}_{\mp}(u),\varphi) &= \mp\int_{\Omega} \tilde{f}_{\mp}(\mathbf{x},u^{\mp}) \varphi\ d\Omega \quad\text{ for all } u,\varphi\in H^{1}_{D}
	\end{aligned}	
	\end{equation}		
	and $\tilde{F}: H^{1}_{D}\mapsto H^{1}_{D}$ as
	\begin{equation}\label{Eq:define F}
	\tilde{F}:= \tilde{F}_{+}+\tilde{F}_{-}.	
	\end{equation}
	\begin{lemma}\label{Lem:properties of beta}
		The operator $\beta$ is positively homogeneous (i.e. $\beta(t u) = t \beta(u)$ for all $t>0, u\in H^{1}_{D}$) and 
		\begin{align}	
		&\text{i)}\quad \exists c\in\mathbb{R}: \lVert \beta(u) \rVert_{H^{1}_{D}} \leq c \lVert s_{-} \rVert_{\infty}\lVert u^{-} \rVert_{H^{1}_{D}} + c \lVert s_{+} \rVert_{\infty}\lVert u^{+} \rVert_{H^{1}_{D}}\quad \forall u\in H^{1}_{D}\label{Eq:beta bounded},\\
		&\text{ii)}\quad u_{n}\rightharpoonup u \implies \beta(u_{n})\rightarrow\beta(u)\label{Eq:beta property},\\
		&\text{iii)}\quad (\beta(u),u) \geq 0\quad \forall u\in H^{1}_{D},\label{Eq:beta inner prod nonnegativity}\\
		&\text{iv)}\quad u_{n}\rightarrow 0, \frac{u_{n}}{\lVert u_{n}\rVert_{H^{1}_{D}}}\rightharpoonup w \implies \frac{\tilde{F}(u_{n})}{\lVert u_{n}\rVert_{H^{1}_{D}}}\rightarrow \beta(w)\label{Eq:beta F relation}.
		\end{align}
	\end{lemma}
	\begin{proof}\mbox{}\\
		\indent The positive homogeneity is apparent.\\
		\begin{enumerate}[i)]
			\item{Using the continuous embedding $H^{1}_{D}\hookrightarrow L^{2}$ and H\"{o}lder's inequality we get
				\begin{equation*}
				\begin{aligned}		
				\lVert \beta(u)\rVert &= \sup_{\lVert\varphi\rVert_{H^{1}_{D}}\leq 1} |(\beta(u),\varphi)| = \sup_{\lVert\varphi\rVert_{H^{1}_{D}}\leq 1} \left| \int_{\Omega} s_{+}u^{+}\varphi\, d\Omega -\int_{\Omega} s_{-}u^{-}\varphi\, d\Omega \right| \leq\\		
				&\leq \lVert s_{+}\rVert_{\infty} \sup_{\lVert\varphi\rVert_{H^{1}_{D}}\leq 1} \left\{\lVert u^{+}\rVert_{L^{2}}\cdot \lVert \varphi\rVert_{L^{2}}\right\} + \lVert s_{-}\rVert_{\infty} \sup_{\lVert\varphi\rVert_{H^{1}_{D}}\leq 1} \left\{\lVert u^{-}\rVert_{L^{2}}\cdot \lVert \varphi\rVert_{L^{2}}\right\} \leq\\
				& \leq c \lVert s_{+}\rVert_{\infty} \sup_{\lVert\varphi\rVert_{H^{1}_{D}}\leq 1} \left\{\lVert u^{+}\rVert_{H^{1}_{D}}\cdot \lVert \varphi\rVert_{H^{1}_{D}}\right\} + c \lVert s_{-}\rVert_{\infty} \sup_{\lVert\varphi\rVert_{H^{1}_{D}}\leq 1} \left\{\lVert u^{-}\rVert_{H^{1}_{D}}\cdot \lVert \varphi\rVert_{H^{1}_{D}}\right\}\leq\nonumber\\
				&\leq c \lVert s_{+}\rVert_{\infty} \lVert u^{+}\rVert_{H^{1}_{D}} + c \lVert s_{-}\rVert_{\infty} \lVert u^{-}\rVert_{H^{1}_{D}}.
				\end{aligned}
				\end{equation*}}
			\item{Let's have a sequence $(u_{n})\subset H^{1}_{D}$ such that $u_{n}\rightharpoonup u\in H^{1}_{D}$. Then by the compact embedding $W^{1,2}\hookrightarrow\hookrightarrow L^{2}$, we get $u_{n}\rightarrow u$ in $L^{2}$. It is easy to see that $|u^{-}_{n}-u^{-}|\leq |u_{n}-u|\ \text{ holds almost everywhere on }\Omega$. Hence,
				\begin{equation*}
				\begin{aligned}
				\lVert \beta^{-}(u_{n})-\beta^{-}(u)\rVert_{H^{1}_{D}} &= \sup_{\lVert \varphi\rVert_{H^{1}_{D}}\leq 1} |(\beta^{-}(u_{n})-\beta^{-}(u),\varphi)| \leq \sup_{\lVert \varphi\rVert_{H^{1}_{D}}\leq 1} \int_{\Omega} |u^{-}_{n}-u^{-}|\cdot |\varphi|\ d\Omega \leq\\
				&\leq C\lVert u_{n}-u\rVert_{L^{2}} \rightarrow 0.
				\end{aligned}
				\end{equation*}
				The same can be shown for $\beta^{+}$ and the assertion follows.}
			\item{Let $u\in H^{1}_{D}$ be arbitrary and $\Omega_{+},\Omega_{-}$ subsets of the domain $\Omega$ such that $\Omega = \Omega_{+}\cup\Omega_{-}$, $u\geq 0$ a.e. on $\Omega_{+}$ and $u<0$ a.e. on $\Omega_{-}$. Hence
				\begin{equation*}
				(\beta(u),u) = \int_{\Omega} s_{+} u^{+}u\; d\Omega - \int_{\Omega} s_{-} u^{-}u\; d\Omega = \int_{\Omega_{+}} s_{+}u^{2} d\Omega_{+} +\int_{\Omega_{-}} s_{-} u^{2} d\Omega_{-}
				\end{equation*}
				and our assertion follows.}
			\item{Now we will define a new auxiliary operator $F: H^{1}_{D}\mapsto H^{1}_{D}$ by
				\begin{equation*}
				(F(u),\varphi) = -\int_{\Omega} (\tilde{f}_{-}(\mathbf{x},u) - s_{-}u)\varphi\ d\Omega\quad \text{ for all }u,\varphi\in H^{1}_{D}.
				\end{equation*}
				We have
				\begin{equation*}
				\lim\limits_{\xi\rightarrow 0} \frac{\tilde{f}_{-}(\mathbf{x},\xi) - s_{-}\xi}{\xi} = 0 \quad\text{ for a.a. }\mathbf{x}\in\Omega
				\end{equation*}
				by assumption $(\ref{Eq:tilda terms cond})$. The growth conditions $(\ref{Eq:tilda f constrains})$ and Proposition 3.2 of \cite{Eisner-Kucera-Vath} give
				\begin{equation}\label{Eq:auxiliary G operator limit}
				\lim\limits_{u\rightarrow 0}\frac{F(u)}{\lVert u\rVert_{H^{1}_{D}}} = 0.
				\end{equation}
				If $u_{n}\rightarrow 0$, then $u^{-}_{n}\rightarrow 0$ (see \cite{Ziemer}) and using $(\ref{Eq:auxiliary G operator limit})$ we get
				\begin{equation*}
				\lim\limits_{n\rightarrow +\infty} \frac{\lVert \tilde{F}_{-}(u_{n}) - \beta^{-}(u_{n}) \rVert_{H^{1}_{D}}}{\lVert u_{n} \rVert_{H^{1}_{D}}} = \lim\limits_{n\rightarrow +\infty} \frac{\lVert F(u^{-}_{n}) \rVert_{H^{1}_{D}}}{\lVert u_{n} \rVert_{H^{1}_{D}}} \leq \lim\limits_{n\rightarrow +\infty} \frac{\lVert F(u^{-}_{n}) \rVert_{H^{1}_{D}}}{\lVert u^{-}_{n} \rVert_{H^{1}_{D}}} = 0.
				\end{equation*}
				If $u_{n}\rightarrow 0$, $\frac{u_{n}}{\lVert u_{n}\rVert_{H^{1}_{D}}}\rightharpoonup w$ then
				\begin{equation*}
				\frac{\tilde{F}_{-}(u_{n})}{\lVert u_{n} \rVert_{H^{1}_{D}}}\rightarrow \beta^{-}(w)
				\end{equation*}
				due to positive homogeneity of $\beta^{-}$ and $(\ref{Eq:beta property})$.\newline
				The same can be shown for $\tilde{F}_{+}$ and $\beta^{+}$ and the assertion is proved.}
		\end{enumerate}	
	\end{proof}
	In order to give an operator formulation of the problem (\ref{Eq:stationary problem}) or (\ref{Eq:stationary problem nonlinear}) with unilateral sources and sinks on the boundary $(\ref{Eq:boundery conditions with unilateral terms})$, we define operators $\beta^{\pm}_{N}: H^{1}_{D}\mapsto H^{1}_{D}$ as
	\begin{equation}\label{Eq:define beta N}
	(\beta^{\mp}_{N}(u),\varphi) = \mp\int_{\Gamma_{N}} s_{\mp}u^{\mp} \varphi\ d\Gamma_{N} \quad\text{ for all } u,\varphi\in H^{1}_{D}
	\end{equation}
	and $\beta_{N}:H^{1}_{D}\mapsto H^{1}_{D}$ as
	\begin{equation}
	\beta_{N}=\beta^{+}_{N}+\beta^{-}_{N}.
	\end{equation}
	\begin{remark}\label{Rem:properties of beta N}
		The operator $\beta_{N}$ possess the same properties as the operator $\beta$ (see Lemma $\ref{Lem:properties of beta}$). 
	\end{remark}
	Let's emphasize that for cases $\Gamma_{D}=\emptyset$ and $\Gamma_{D}\neq\emptyset$ we have two different inner products and therefore operators defined above are in these two cases also different. In the case $\Gamma_{D}\neq\emptyset$ we consider the function space $H^{1}_{D}$ equipped with the inner product $(u,\varphi) =  \int_{\Omega} \nabla u \nabla \varphi \ d\Omega$. A weak solution of the problem $(\ref{Eq:stationary problem with unilateral terms})$,$(\ref{Eq:zero boundary conditions})$ or $(\ref{Eq:stationary problem nonlinear with unilateral terms})$,$(\ref{Eq:zero boundary conditions})$ is then a pair of functions $u,v\in H^{1}_{D}$ satisfying
	\begin{equation}\label{Eq:operator formulation unilateral}
	\begin{aligned}
	d_{1}u -b_{1,1}Au - b_{1,2}Av + \beta(u) &= 0,\\
	d_{2}v -b_{2,1}Au - b_{2,2}Av &= 0
	\end{aligned}
	\end{equation}
	or
	\begin{equation}\label{Eq:operator formulation unilateral nonlinear}
	\begin{aligned}
	d_{1}u -b_{1,1}Au - b_{1,2}Av - N_{1}(u,v) + \tilde{F}(u) &= 0,\\
	d_{2}v -b_{2,1}Au - b_{2,2}Av - N_{2}(u,v) &= 0,	
	\end{aligned}
	\end{equation}
	respectively.\\
	\indent If $\Gamma_{D} = \emptyset$, the function space $H^{1}_{D}$ is identical with $W^{1,2}$ and is equipped with the inner product $(u,\varphi) = \int_{\Omega} \left(\nabla u \nabla \varphi + u \varphi\right) \ d\Omega$. A weak solution of $(\ref{Eq:stationary problem with unilateral terms})$,$(\ref{Eq:zero boundary conditions})$ or $(\ref{Eq:stationary problem nonlinear with unilateral terms})$,$(\ref{Eq:zero boundary conditions})$ is then a pair of functions $u,v\in W^{1,2}$ satisfying
	\begin{equation}\label{Eq:operator formulation unilateral Neumann}
	\begin{aligned}
	d_{1}(I-A)u -b_{1,1}Au -b_{1,2}Av +\beta(u) &=& 0,\\
	d_{2}(I-A)v -b_{2,1}Au -b_{2,2}Av &=& 0
	\end{aligned}
	\end{equation}
	or
	\begin{equation}\label{Eq:operator formulation unilateral noonlinear Neumann}
	\begin{aligned}
	d_{1}(I-A)u -b_{1,1}Au -b_{1,2}Av -N_{1}(u,v) +\tilde{F}(u) &=& 0,\\
	d_{2}(I-A)v -b_{2,1}Au -b_{2,2}Av -N_{2}(u,v) &=& 0,
	\end{aligned}
	\end{equation}
	respectively.\\
	\indent For the problem $(\ref{Eq:stationary problem}),(\ref{Eq:boundery conditions with unilateral terms})$ or $(\ref{Eq:stationary problem nonlinear}),(\ref{Eq:boundery conditions with unilateral terms})$ we will get analogous systems, we just replace operators $\beta$ and $\tilde{F}$ with $\beta_{N}$.

	\section{Critical points for fixed $d_{2}$}\label{Sec:fixed d2}
	In this Section we will assume that $d_{2}>0$ is fixed and we will use the notation from Sections $\ref{Sec:basic ass a def}$ and $\ref{Sec:abstract fomulation}$. As usually, by an eigenvalue of a positively homogeneous operator $P$ we mean a number $\lambda$ such that the equation $P(u)=\lambda u$ has a non-trivial solution. More generally, by an eigenvalue of a problem with a positively homogeneous operator we mean a parameter for which the problem under consideration has a non-trivial solution.
	
	\subsection{Reduction to one operator equation for the case $\Gamma_{D}\neq\emptyset$}\label{Sec:Problem with Dirichlet and Neumann boundary conditions}
	Let's suppose $\Gamma_{D}\neq\emptyset$. Since the operator $A$ is positive by Remark \ref{Rem:eigenvalues of A} and $b_{2,2}<0$ by the assumption ($\ref{Eq:B assumption}$), the number $\frac{d_{2}}{b_{2,2}}$ is not its eigenvalue. Therefore the operator $d_{2}I-b_{2,2}A$ is invertible and surjective. Hence, we can express $v$ from the second equation in $(\ref{Eq:operator formulation unilateral})$, substitute it into the first one and get
	\begin{equation*}
	d_{1}u - b_{1,1}Au  -b_{1,2}A(d_{2}I - b_{2,2}A)^{-1}b_{2,1}Au + \beta(u) = 0.
	\end{equation*}
	Introducing the operator $S_{d_2}: H^{1}_{D}\mapsto H^{1}_{D}$ as
	\begin{equation}\label{Eq:define S}
	S_{d_2}:= b_{1,1}A + b_{1,2}A(d_{2}I - b_{2,2}A)^{-1}b_{2,1}A,
	\end{equation}
	we can write the system $(\ref{Eq:operator formulation unilateral})$ as
	\begin{subequations}\label{Eq:system with extracted v and S unilateral}
	\begin{eqnarray}
	d_{1}u -S_{d_{2}}u + \beta(u) = 0,\label{Eq:reduced operator equation unilateral}\\
	v = (d_{2}I - b_{2,2}A)^{-1}b_{2,1}Au.
	\end{eqnarray}	
	\end{subequations}
	In particular, the system of the operator equations
	\begin{subequations}\label{Eq:system with extracted v and S}
	\begin{eqnarray}
		d_{1}u -S_{d_{2}}u = 0,\label{Eq:reduced operator equation}\\
		v = (d_{2}I - b_{2,2}A)^{-1}b_{2,1}Au
	\end{eqnarray}
	\end{subequations}
	is equivalent with the system
	\begin{equation}\label{Eq:operator system}
		\begin{aligned}
		d_{1}u -b_{1,1}Au - b_{1,2}Av &= 0,\\
		d_{2}v -b_{2,1}Au - b_{2,2}Av &= 0.	
		\end{aligned}
	\end{equation}
	\begin{remark}\label{Rem:sharing eigen of S}
		The operator $S_{d_{2}}: H^{1}_{D}\mapsto H^{1}_{D}$ defined by $(\ref{Eq:define S})$ is linear, bounded, symmetric and compact. It follows from simple calculations and Remark \ref{Rem:eigenvalues of A} that the eigenvalues of the operator $S_{d_{2}}$ are
		\begin{equation}\label{Eq:eigenvalues of S Dirichlet}
		d^{j}_{1}=\frac{1}{\kappa_{j}}\left(\frac{b_{1,2}b_{2,1}}{d_{2}\kappa_{j}-b_{2,2}} + b_{1,1}\right),\; j=1,2,\ldots
		\end{equation}
		and since $\kappa_{j}\rightarrow\infty$ as $j\rightarrow\infty$, we get $d^{j}_{1}\rightarrow 0$ as $j\rightarrow\infty$. The eigenvectors of $S_{d_{2}}$ corresponding to $d^{j}_{1}$ coincide with those of the operator $A$ corresponding to $\mu_{j}$, i.e. with the eigenfunctions of $(\ref{Eq:Laplace eigenvalue problem})$ corresponding to $\kappa_{j}$.
	\end{remark}		
	
	\subsection{Reduction to one operator equation for the case $\Gamma_{D}=\emptyset$}\label{Sec:Problem with pure Neumann boundary conditions}
	\indent Let's consider the case $\Gamma_{D}=\emptyset$. It follows from Remark 	\ref{Rem:eigenvalues of A} that the number $d_{2}$ is not an eigenvalue of the operator $d_{2}A + b_{2,2}A$. Indeed, we have $d_{2}\neq \frac{d_{2}+b_{2,2}}{\kappa_{j}+1}$, because $d_{2}\kappa_{j}\neq b_{2,2}$ ($b_{2,2}$ is negative by $(\ref{Eq:B assumption})$). Hence, the operator $d_{2}I -d_{2}A - b_{2,2}A$ (in $(\ref{Eq:operator formulation unilateral Neumann})$) is surjective and invertible. Similarly as in Section $\ref{Sec:Problem with Dirichlet and Neumann boundary conditions}$ we can transform the system $(\ref{Eq:operator formulation unilateral Neumann})$ to the system
	\begin{subequations}\label{Eq:system with extracted v and S unilateral Neumann}
		\begin{eqnarray}
		d_{1}(I-A)u -S_{d_{2}}u + \beta(u) = 0,\label{Eq:reduced operator equation unilateral Neumann}\\
		v = (d_{2}I -d_{2}A - b_{2,2}A)^{-1}b_{2,1}Au,
		\end{eqnarray}
	\end{subequations}
	with the new operator
	\begin{equation}\label{Eq:define S Neumann}
	S_{d_{2}}:= b_{1,1}A + b_{1,2}A(d_{2}I -d_{2}A - b_{2,2}A)^{-1}b_{2,1}A.
	\end{equation}
	In particular, the system of the operator equations
	\begin{subequations}\label{Eq:system with extracted v and S Neumann}
		\begin{eqnarray}
		d_{1}(I-A)u -S_{d_{2}}u = 0,\label{Eq:reduced operator equation Neumann}\\
		v = (d_{2}I -d_{2}A - b_{2,2}A)^{-1}b_{2,1}Au
		\end{eqnarray}
	\end{subequations}
	is equivalent with the system
	\begin{equation}\label{Eq:operator system Neumann}
	\begin{aligned}
	d_{1}(I-A)u -b_{1,1}Au - b_{1,2}Av &= 0,\\
	d_{2}(I-A)v -b_{2,1}Au - b_{2,2}Av &= 0.	
	\end{aligned}
	\end{equation}
	\begin{remark}\label{Rem:eigenvalues of S}
		The operator $S_{d_{2}}$ defined by $(\ref{Eq:define S Neumann})$ is linear, continuous, symmetric and compact. Simple calculations and Remark \ref{Rem:eigenvalues of A} imply that the eigenvalues of the operator $S_{d_{2}}$ are 
		\begin{equation}\label{Eq:eigenvalues of S Neumann}
		\lambda^{j} = \frac{1}{\kappa_{j}+1}\left(\frac{b_{1,2}b_{2,1}}{d_{2}\kappa_{j}-b_{2,2}} + b_{1,1}\right),\; j=0,1,2,\ldots 		
		\end{equation}	
		and the eigenvectors of $S_{d_{2}}$ corresponding to $\lambda^{j}$ coincide with those of $A$ corresponding to $\mu_{j}$, i.e. with the eigenfunctions of $(\ref{Eq:Laplace eigenvalue problem})$ corresponding to $\kappa_{j}$. However, the eigenvalues $d^{j}_{1}$ of the problem $(\ref{Eq:reduced operator equation Neumann})$ are the same as those of the operator 
		$S_{d_{2}}$ defined by (\ref{Eq:define S}) in the case $\Gamma_D\ne \emptyset$, i.e. they are given by (\ref{Eq:eigenvalues of S Dirichlet}). (There is no eigenvalue with $j=0$.)
	\end{remark}	
		
	\subsection{Maximal eigenvalues and critical points}

\begin{notation}\label{Not:denoting maximal eigenvalues}
		We will denote by $d^{MAX}_{1}$ the maximal eigenvalue of the operator $S_{d_{2}}$ or of the problem $(\ref{Eq:reduced operator equation Neumann})$ in the case $\Gamma_{D}\neq\emptyset$ or $\Gamma_{D}=\emptyset$, respectively. We will also denote by $d^{MAX,\beta}_{1}$ the maximal eigenvalue of the operator $S_{d_{2}}-\beta$ or of the problem $(\ref{Eq:reduced operator equation unilateral Neumann})$ in the case $\Gamma_{D}\neq\emptyset$ or $\Gamma_{D}=\emptyset$, respectively, if it exists.
	\end{notation}
	
	\begin{observation}\label{Ob:maximal eigenvalue of S}
		We can see from the form (\ref{Eq:eigenvalues of S Dirichlet}) of the eigenvalues $d^{j}_{1}$ (see Remarks \ref{Rem:sharing eigen of S}, \ref{Rem:eigenvalues of S}) and from $(\ref{Eq:hyperbolas})$, that a point $[d_{1},d_{2}]$ lies on a hyperbola $C_{j}$ for some $j\in\mathbb{N}$ if and only if $d_{1}$ is an eigenvalue of $S_{d_{2}}$ in the case $\Gamma_D\ne\emptyset$ or an eigenvalue of (\ref{Eq:reduced operator equation Neumann}) in the case $\Gamma_D=\emptyset$. For the maximal eigenvalue $d^{MAX}_{1}$ of $S_{d_{2}}$ in the case $\Gamma_D\ne\emptyset$ and of (\ref{Eq:reduced operator equation Neumann}) in the case $\Gamma_D=\emptyset$ we have $[d^{MAX}_{1},d_{2}]\in C_{E}$. It follows from $(\ref{Eq:eigenvalues of S Dirichlet})$ and Remark \ref{Rem:eigenvalues of S} that the operator $S_{d_{2}}$ in the case $\Gamma_D\ne\emptyset$ and the problem (\ref{Eq:reduced operator equation Neumann}) in the case $\Gamma_D=\emptyset$ have infinitely many positive eigenvalues and maximally finite number of negative eigenvalues. See also Figure $\ref{Fig:hyperbolas}$.
	\end{observation}	
		
	\begin{lemma}\label{Lem:critical points vs eigenvalues}
	If $\Gamma_D\ne\emptyset$, then a point $[d_{1},d_{2}]\in\mathbb{R}^{2}_{+}$ is a critical point of the system $(\ref{Eq:stationary problem}),(\ref{Eq:zero boundary conditions})$ or $(\ref{Eq:stationary problem with unilateral terms}),(\ref{Eq:zero boundary conditions})$ if and only if $d_{1}$ is an eigenvalue of the operator $S_{d_2}$ or $S_{d_2}-\beta$, respectively.\newline	
	\indent If $\Gamma_D=\emptyset$, then a point $[d_{1},d_{2}]\in\mathbb{R}^{2}_{+}$ is a critical point of the system $(\ref{Eq:stationary problem}),(\ref{Eq:zero boundary conditions})$ or $(\ref{Eq:stationary problem with unilateral terms}),(\ref{Eq:zero boundary conditions})$ if and only if $d_{1}$ is an eigenvalue of the problem (\ref{Eq:reduced operator equation Neumann}) or (\ref{Eq:reduced operator equation unilateral Neumann}), respectively.
	\end{lemma}
	\begin{proof}
		Let $\Gamma_{D}\neq\emptyset$. A point $[d_{1},d_{2}]\in\mathbb{R}^{2}_{+}$ is a critical point of the system $(\ref{Eq:stationary problem}),(\ref{Eq:zero boundary conditions})$ or $(\ref{Eq:stationary problem with unilateral terms}),(\ref{Eq:zero boundary conditions})$ if and only if there exists a non-trivial solution $[u,v]$ of $(\ref{Eq:operator system})$ or $(\ref{Eq:operator formulation unilateral})$, respectively. This is true if and only if there exists a non-trivial solution $u\in H^{1}_{D}$ of $(\ref{Eq:reduced operator equation})$ or $(\ref{Eq:reduced operator equation unilateral})$, i.e. $d_{1}$ is an eigenvalue of the operator $S_{d_2}$ or $S_{d_2}-\beta$, respectively (see Section \ref{Sec:Problem with Dirichlet and Neumann boundary conditions}). The proof for the case $\Gamma_{D}=\emptyset$ is analogous, we only use (\ref{Eq:operator system Neumann}), $(\ref{Eq:operator formulation unilateral Neumann})$, $(\ref{Eq:reduced operator equation Neumann})$, $(\ref{Eq:reduced operator equation unilateral Neumann})$ and the result of Section \ref{Sec:Problem with pure Neumann boundary conditions}.
	\end{proof}
	We will use a variational characterization of the largest eigenvalue of an eigenvalue problem with a positively homogeneous operator to a study of critical points of the problem $(\ref{Eq:stationary problem with unilateral terms}),(\ref{Eq:zero boundary conditions})$. The following abstract theorem is a slight modification of the result proved for the particular case $L\equiv 0$ in \cite{Kucera-Navratil-Dirichlet} and for the general case in a forthcoming paper of J. Navr\' atil. Let's remind that $Ker(I-L)$ is the kernel of the operator $I-L$.
	\begin{theorem}\label{Thm:Navratil theorem}
		Let $H$ be a Hilbert space, $P: H\mapsto H$ a positively homogeneous, continuous operator such that 
		\begin{equation*}
			u_{n}\rightharpoonup u \implies P(u_{n})\rightarrow P(u)
		\end{equation*}
		and $L: H\mapsto H$ a linear, continuous, symmetric and compact operator. In the case $L\not\equiv 0$ we suppose that the maximal eigenvalue of $L$ is in the interval $\left(0,1 \right]$. Let there exist $u_{0}\in H, u_{0}\notin Ker(I-L)$ such that
		\begin{equation}\label{E Navratil max lambda}
		\lambda_{0} := \max\limits_{\substack{u\in H \\ u\notin Ker(I-L)}} \frac{(P(u),u)}{((I-L)u,u)} = \frac{(P(u_{0}),u_{0})}{((I-L)u_{0},u_{0})}>0
		\end{equation}
		and
		\begin{equation}\label{Eq:Navratil limit condition}
		\lim\limits_{t\rightarrow 0} \frac{1}{t}(P(u_{0}+th)-P(u_{0}),u_{0}) = (P(u_{0}),h)\quad\forall h\in H.
		\end{equation}
		Then $\lambda_{0}$ is the maximal eigenvalue of the problem 
		\begin{equation}\label{Eq:general eigenvalue problem}
		\lambda(I-L)u - P(u) = 0
		\end{equation}
		and $u_0$ is a corresponding eigenvector. If $u_1\notin Ker(I-L)$ is an arbitrary eigenvector of $(\ref{Eq:general eigenvalue problem})$ corresponding to $\lambda_0$ then it satisfies $(\ref{E Navratil max lambda})$ with $u_0$ replaced by $u_1$. 
	\end{theorem}
	
	Let us note that the problem (\ref{Eq:general eigenvalue problem}) has an eigenvector in $Ker(I-L)$ only if there is $u\in Ker(I-L)$ such that $P(u)=0$. In this case any $\lambda$ is an eigenvalue.
	
	\begin{proof}\mbox{}\\
		We will assume that $L\not\equiv 0$, the case $L\equiv 0$ is simpler. Let us denote by $\mu_L^{MAX}$ the maximal eigenvalue of $L$. Since $\mu_L^{MAX} \in \left(0,1\right]$, we have $\max\limits_{\substack{u\neq 0 \\ \lVert u\rVert_{H^{1}_{D}} = 1}} (Lu,u) = \mu_L^{MAX}\leq 1$. If $\mu_L^{MAX} <1$, then $(Lu,u)<1$  and therefore $((I-L)u,u)>0$ for all $u$. If $\mu_L^{MAX}=1$, then $\max\limits_{\substack{u\neq 0 \\ \lVert u\rVert_{H^{1}_{D}} = 1}} (Lu,u) = 1$, but the maximum is attained only in the elements of $Ker(I-A)$. Hence, $((I-L)u,u)>0$ for all $u\notin Ker(I-L)$ and the expression in $(\ref{E Navratil max lambda})$ makes sense.\\		
		\indent Let $u_{0}\notin Ker(I-L)$ be arbitrary such that (\ref{E Navratil max lambda}) and (\ref{Eq:Navratil limit condition}) are fulfilled, and let $h\in H^{1}_{D}$ be arbitrary fixed. Then for $t\in\mathbb{R}$ small such that 
		$(u_{0}+th) \notin Ker(I-L)$ we have
		\begin{equation*}
		\frac{(P(u_{0}+th),u_{0}+th)}{((I-L)(u_{0}+th),u_{0}+th)} \leq \frac{(P(u_{0}),u_{0})}{((I-L)u_{0},u_{0})} =: \lambda_{0}.
		\end{equation*}
		We can rewrite this inequality as
		\begin{equation*}
		\begin{aligned}
		(P(u_{0}+th),u_{0}) + t(P(u_{0}+th),h) \leq \frac{(P(u_{0}),u_{0})}{((I-L)u_{0},u_{0})}\left[((I-L)u_{0},u_{0}) + 2t((I-L)u_{0},h) + t^2 ((I-L)h,h)\right]
		\end{aligned}	
		\end{equation*}
		and eventually as
		\begin{equation*}
		\begin{aligned}
		(P(u_{0}+th),u_{0})-(P(u_{0}),u_{0}) + t(P(u_{0}+th),h) \leq \lambda_{0}\left[2t((I-L)u_{0},h) + t^2 ((I-L)h,h)\right].
		\end{aligned}	
		\end{equation*}
		We divide it by $2t$ and get
		\begin{equation*}
		\begin{aligned}
		\frac{1}{2t}\left[(P(u_{0}+th),u_{0})-(P(u_{0}),u_{0})\right] + \frac{1}{2}(P(u_{0}+th),h)
		\leq \lambda_{0}\left[((I-L)u_{0},h) + \frac{t}{2} ((I-L)h,h)\right],\quad t>0,\\
		\frac{1}{2t}\left[(P(u_{0}+th),u_{0})-(P(u_{0}),u_{0})\right] + \frac{1}{2}(P(u_{0}+th),h)
		\geq \lambda_{0}\left[((I-L)u_{0},h) + \frac{t}{2} ((I-L)h,h)\right],\quad t<0.
		\end{aligned}	
		\end{equation*}	
		Let $t\rightarrow 0$. We use the condition $(\ref{Eq:Navratil limit condition})$ and continuity of $P$ to get
		\begin{equation*}
		\begin{aligned}
		(P(u_{0}),h)\leq \lambda_{0}((I-L)u_{0},h),\\
		(P(u_{0}),h)\geq \lambda_{0}((I-L)u_{0},h).
		\end{aligned}	
		\end{equation*}
		Since $h$ was arbitrary, we have
		\begin{equation*}
		(P(u_{0}),h) = \lambda_{0}((I-L)u_{0},h)\quad \text{ for all } h\in H^{1}_{D},
		\end{equation*}
		that means
		\begin{equation*}
		P(u_{0}) = \lambda_{0}(I-L)u_{0}.
		\end{equation*}
		Hence, the number $\lambda_{0}$ is an eigenvalue of the problem (\ref{Eq:general eigenvalue problem}) and $u_{0}$ is a corresponding eigenvector. 
		\indent	Let $\lambda_{1}$ be another eigenvalue of the problem (\ref{Eq:general eigenvalue problem}) and let $u_{1}\notin Ker(I-L)$ be a corresponding eigenvector. Then we have
		\begin{equation*}
		P(u_{1}) = \lambda_{1}(I-L)u_{1}
		\end{equation*}
		and if we multiply it by $u_{1}$ and divide by $((I-L)u_{1},u_{1})$, we get
		\begin{equation*}
		\lambda_{1} = \frac{(P(u_{1}),u_{1})}{((I-L)u_{1},u_{1})} \leq \frac{(P(u_{0}),u_{0})}{((I-L)u_{0},u_{0})} = \lambda_{0}.
		\end{equation*}
		Hence, $\lambda_{0}$ is the maximal eigenvalue. 
		If $\lambda_{1}=\lambda_{0}$, then we have equality in the last estimate, that means $u_{1}$ is a maximizer of the expression (\ref{E Navratil max lambda}). That means an arbitrary eigenvector corresponding to $\lambda_{0}$ not lying in $Ker(I-L)$ satisfies (\ref{E Navratil max lambda}) with $u_0$ replaced by $u_1$.
	\end{proof}	
	\noindent If the condition $(\ref{Eq:Navratil limit condition})$ is fulfilled for any $u_{0}$, then it actually means that $P$ has a potential $\varPhi=\frac{1}{2}(Pu,u)$. 
	\begin{remark}\label{Rem:zero operator L}
		In the particular case $L\equiv 0$, $\lambda_{0} := \max\limits_{\substack{u\in H \\ u\neq o}} \frac{(P(u),u)}{\lVert u\rVert^{2}_{H}}$ is the maximal eigenvalue of $P$.
	\end{remark}	
	\begin{theorem}\label{Thm:maximal eigenvalues of S-beta}
		Let $\Gamma_{D}\neq\emptyset$ and let $S_{d_{2}}$ be the operator from $(\ref{Eq:define S})$. If there exists a function $\varphi\in H^{1}_{D}$ such that
		\begin{equation}\label{Eq:positive supremum cond}
			(S_{d_{2}}\varphi,\varphi)-(\beta(\varphi),\varphi)>0,
		\end{equation}
		then the maximal eigenvalue of the operator $S_{d_{2}}-\beta$ is
		\begin{equation}\label{Eq:maximal eigenvalue of S-beta}
		d^{MAX,\beta}_{1} := \max\limits_{\substack{u\in H^1_D \\ u\neq o}} \frac{(S_{d_{2}}u,u)-(\beta(u),u)}{\lVert u\rVert^{2}_{H^{1}_{D}}} = \max\limits_{\substack{u\in H^1_D \\ \lVert u\rVert_{H^{1}_{D}}=1}} (S_{d_{2}}u,u)-(\beta(u),u)>0.
		\end{equation}
		Maximizers of the expression in $(\ref{Eq:maximal eigenvalue of S-beta})$ are exactly all eigenvectors of $S_{d_{2}}-\beta$ corresponding to $d^{MAX,\beta}_{1}$.
	\end{theorem}
	\begin{proof}\mbox{}\\
		Let's prove that the maximum in $(\ref{Eq:maximal eigenvalue of S-beta})$ exists. Let
		\begin{equation*}
		M := \sup\limits_{\substack{u\in H^1_D \\ \lVert u\rVert_{H^{1}_{D}}=1}} (S_{d_{2}}u,u)-(\beta(u),u).
		\end{equation*}
		The existence of $\varphi$ satisfying $(\ref{Eq:positive supremum cond})$ implies $M>0$. We can choose a sequence $(u_n)\subset H^{1}_{D}$ with $\lVert u_n \rVert_{H^{1}_{D}}=1$ such that
		\begin{equation}
		\lim\limits_{n\rightarrow\infty} (S_{d_{2}}u_n,u_n)-(\beta(u_n),u_n)= M.
		\end{equation}
		We can assume $u_{n}\rightharpoonup u_{0}\in H^{1}_{D}$. Since $S_{d_{2}}$ is linear and compact and $\beta$ satisfies $(\ref{Eq:beta property})$, we get
		\begin{equation}
		(S_{d_{2}}u_n,u_n)-(\beta(u_n),u_n)\rightarrow (S_{d_{2}}u_0,u_0)-(\beta(u_0),u_0) = M.
		\end{equation}
		Now we will show that $\lVert u_{0}\rVert_{H^{1}_{D}} = 1$. We know that $\lVert u_{0}\rVert_{H^{1}_{D}} \leq 1$. If $0<\lVert u_{0}\rVert_{H^{1}_{D}} < 1$, then $\left(S_{d_{2}}\frac{u_0}{\lVert u_{0}\rVert_{H^{1}_{D}}},\frac{u_0}{\lVert u_{0}\rVert_{H^{1}_{D}}}\right)-\left(\beta\left(\frac{u_0}{\lVert u_{0}\rVert_{H^{1}_{D}}}\right),\frac{u_0}{\lVert u_{0}\rVert_{H^{1}_{D}}}\right) = \frac{M}{\lVert u_{0}\rVert^{2}_{H^{1}_{D}}} > M$ due to positive homogeneity of $\beta$ (see Lemma \ref{Lem:properties of beta}), which contradicts the fact that $M$ is supremum. If $u_{0}=0$, then $M=0$, which is not the case. Therefore the last maximum in $(\ref{Eq:maximal eigenvalue of S-beta})$ exists and it is attained at $u_0$ with $\lVert u_0 \rVert_{H^{1}_{D}} = 1$. The equality between two maxima in $(\ref{Eq:maximal eigenvalue of S-beta})$ follows from the positive homogeneity of $\beta$.\\		
		\indent It is known that $P=\beta$ and therefore also $P=S_{d_{2}}-\beta$ satisfies $(\ref{Eq:Navratil limit condition})$ for any $u_0$ (see Lemma $\ref{Lem:Navratil conditions}$ in Appendix). The operator $P= S_{d_{2}}-\beta$ satisfies also the other assumptions of Theorem $\ref{Thm:Navratil theorem}$ (see Remark $\ref{Rem:sharing eigen of S}$ and Lemma $\ref{Lem:properties of beta}$). Hence, the assertions of Theorem \ref{Thm:maximal eigenvalues of S-beta} follow from Theorem $\ref{Thm:Navratil theorem}$, where we choose $L=0$, that means we have $Ker(I-L)=\{0\}$.
	\end{proof}
		
	\begin{remark}\label{Rem:Suu<0}	
		Let's consider the case $\Gamma_{D}=\emptyset$. The definition of the inner product and of the operator $A$ (Section 
		$\ref{Sec:abstract fomulation}$) give $((I-A)u,\varphi)=\int_{\Omega} (\nabla u,\nabla \varphi) \; d\Omega$ for all $u, \varphi$. It follows that $(I-A)u=0$ is equivalent to $((I-A)u,u)=0$, and this holds if and only if $u$ is a constant function. In other words, $Ker(I-A)=span\{e_0\}$, $e_0$ being the eigenfunction of (\ref{Eq:Laplace eigenvalue problem}) corresponding to $\kappa_0$. Due to Remark \ref{Rem:eigenvalues of S}, any non-trivial $u_{0}\in Ker(I-A)$ is simultaneously an eigenvector of $S_{d_{2}}$ from $(\ref{Eq:define S Neumann})$ corresponding to $\lambda^{0}$. Hence, by using (\ref{Eq:B assumption}) we get
		\begin{equation}\label{Eq:negativity of S for u constant}
		(S_{d_{2}} u_0, u_0) = (\lambda^{0} u_0, u_0) = \left( b_{1,1} +\frac{b_{1,2}b_{2,1}}{-b_{2,2}} \right)\lVert u_0 \rVert^{2}_{W^{1,2}} = \frac{-det(\mathbf{B})}{-b_{2,2}}\lVert u_0 \rVert^{2}_{W^{1,2}} < 0.
		\end{equation}
	\end{remark}
		
	\begin{theorem}\label{Thm:existence of maximum Neumann}
		Let $\Gamma_{D}=\emptyset$ and let $S_{d_{2}}$ be the operator from $(\ref{Eq:define S Neumann})$. If there exists a function $\varphi\in W^{1,2}$ satisfying $(\ref{Eq:positive supremum cond})$, then the maximal eigenvalue of the problem $(\ref{Eq:reduced operator equation unilateral Neumann})$ is
		\begin{equation}\label{Eq:maximal eigenvalue of S-beta Neumann}
		d^{MAX,\beta}_{1} := \max\limits_{\substack{u\in W^{1,2}\\u\notin Ker(I-A)}}\frac{(S_{d_{2}}u,u)-(\beta(u),u)}{((I-A)u,u)}>0.
		\end{equation}
		Maximizers of the expression in $(\ref{Eq:maximal eigenvalue of S-beta Neumann})$ are exactly all eigenvectors of the problem $(\ref{Eq:reduced operator equation unilateral Neumann})$ corresponding to $d^{MAX,\beta}_{1}$.
	\end{theorem}
	\begin{proof}\mbox{}\\
		Let's denote
		\begin{equation*}
			M:= \sup\limits_{\substack{u\in W^{1,2}\\u\notin Ker(I-A)}}\frac{(S_{d_{2}}u,u)-(\beta(u),u)}{((I-A)u,u)}.
		\end{equation*}
		Since $((I-A)u,u) = \int_{\Omega} \left(\nabla u\right)^2 \; d\Omega \geq 0$ for every $u$ and we assume that there exists a function $\varphi$ satisfying $(\ref{Eq:positive supremum cond})$, we have $M>0$.\\
		\indent We can choose a sequence $u_{n}\notin Ker(I-A)$ with $\lVert u_{n} \rVert_{W^{1,2}} = 1$ such that
		\begin{equation*}
			\lim\limits_{n\rightarrow\infty} \frac{(S_{d_{2}} u_n,u_n)-(\beta(u_n),u_n)}{((I-A)u_n,u_n)} = M.
		\end{equation*}
		We can assume that $u_{n}\rightharpoonup u_{0}$. If $u_{0}= 0$, then we have
		\begin{equation*}
		((I-A)u_{n},u_{n}) = 1-(Au_{n},u_{n})\rightarrow 1-(Au_{0},u_{0}) = 1
		\end{equation*}
		due to the compactness of $A$, and
		\begin{equation*}
		(S_{d_{2}} u_n,u_n)-(\beta(u_n),u_n)\rightarrow (S_{d_{2}} u_0,u_0)-(\beta(u_0),u_0) = 0
		\end{equation*}
		by the compactness of $S_{d_{2}}$ and $(\ref{Eq:beta property})$. This means that $M=0$, which contradicts the positivity of $M$.\\
		\indent Further, let's show that $u_{0}\notin Ker(I-A)\setminus\{0\}$, i.e. $u_{0}$ is not a constant function. Let $u_{0}$ be a non-zero constant function. Then $(S_{d_{2}} u_0, u_0) <0$ by Remark \ref{Rem:Suu<0}. Since we have $-(\beta(u),u)\leq 0$ for every $u$ by $(\ref{Eq:beta inner prod nonnegativity})$, we get $(S_{d_{2}} u_0,u_0)-(\beta(u_0),u_0)<0$ and consequently
		\begin{equation*}
		\lim\limits_{n\rightarrow\infty} \frac{(S_{d_{2}} u_n,u_n)-(\beta(u_n),u_n)}{((I-A)u_n,u_n)} \leq 0.
		\end{equation*}
		That contradicts the fact that $u_{n}$ is a maximizing sequence and the supremum $M$ is positive. Hence, we have $u_{0}\notin Ker(I-L)$.\\
		\indent We need to show that $\lVert u_{0} \rVert_{W^{1,2}} = 1$. We already know that $0<\lVert u_{0} \rVert_{W^{1,2}} \leq 1$. Now let $0<\lVert u_{0} \rVert_{W^{1,2}} < 1$. We have $1-(A u_0,u_0)>0$ (see Remark \ref{Rem:eigenvalues of A}) and 
		\begin{equation*}
		 \frac{(S_{d_{2}} u_n,u_n)-(\beta(u_n),u_n)}{((I-A)u_n,u_n)}\rightarrow  \frac{(S_{d_{2}} u_0,u_0)-(\beta(u_0),u_0)}{1-(A u_0,u_0)}=M
		\end{equation*}
		by the compactness of $S_{d_{2}},A$ and the condition $(\ref{Eq:beta property})$.  Simultaneously $\lVert u_{0} \rVert^{2}_{W^{1,2}}-(A u_0,u_0)>0$ because of $u_{0}\notin Ker(I-A)$ (see Remarks \ref{Rem:eigenvalues of A} and \ref{Rem:Suu<0}). It follows that
		\begin{equation*}
		\frac{(S_{d_{2}} u_0,u_0)-(\beta(u_0),u_0)}{\lVert u_{0} \rVert^{2}_{W^{1,2}}-(A u_0,u_0)} > \frac{(S_{d_{2}} u_0,u_0)-(\beta(u_0),u_0)}{1-(A u_0,u_0)} = M>0,
		\end{equation*}
		which contradicts that fact that $M$ is a supremum. Hence, we have $\lVert u_{0} \rVert_{W^{1,2}} = 1$.\\		
		\indent We use compactness of $S_{d_{2}},A$, the property $(\ref{Eq:beta property})$ of $\beta$ and the fact that $\lVert u_{n} \rVert_{W^{1,2}} = 1 = \lVert u_{0} \rVert_{W^{1,2}}$ to get
		\begin{equation}
		\frac{(S_{d_{2}}u_n,u_n)-(\beta(u_n),u_n)}{((I-A)u_n,u_n)}\rightarrow\frac{(S_{d_{2}}u_0,u_0)-(\beta(u_0),u_0)}{((I-A)u_0,u_0)}.
		\end{equation}
		Hence, the maximum exists and it is attained at the function $u_0 \notin Ker(I-A)$ with $\lVert u_{0} \rVert_{W^{1,2}} = 1$.\\
		\indent  It is know that $P=S_{d_{2}}-\beta$ satisfies $(\ref{Eq:Navratil limit condition})$ for any $u_0$ (see Lemma $\ref{Lem:Navratil conditions}$ in Appendix). The operators $P= S_{d_{2}}-\beta$ and $L= A$ also satisfy the other assumptions of Theorem $\ref{Thm:Navratil theorem}$ (see Remark $\ref{Rem:eigenvalues of A}$, Remark $\ref{Rem:eigenvalues of S}$ and Lemma $\ref{Lem:properties of beta}$). Hence, $d^{MAX,\beta}_{1}$ is the maximal eigenvalue and $u_0$ is a corresponding eigenvector of the problem $(\ref{Eq:reduced operator equation unilateral Neumann})$ by Theorem $\ref{Thm:Navratil theorem}$.\\
		\indent Let's show that if $u_{1}$ is an arbitrary eigenvector of (\ref{Eq:reduced operator equation unilateral Neumann}) corresponding to $d^{MAX,\beta}_{1}$ then $u_{1}\notin Ker(I-A)$. If $u_{1}\in Ker(I-A)\setminus\{0\}$, then we have $-(S_{d_{2}}u_{1},u_{1})+(\beta(u_{1}),u_{1})>0$ (see $(\ref{Eq:negativity of S for u constant}$) and Lemma $\ref{Lem:properties of beta}$) and $d^{MAX,\beta}_{1}((I-A)u_{1},u_{1})=0$, which contradicts the equation (\ref{Eq:reduced operator equation unilateral Neumann}) with $u=u_1$ multiplied by $u_1$. Hence, $u_{1}\notin Ker(I-A)$, and the last assertion of Theorem \ref{Thm:existence of maximum Neumann} follows also from Theorem $\ref{Thm:Navratil theorem}$.\\		
	\end{proof}
	\begin{remark}\label{Rem:existence of maxima}\mbox{}
		The assumption $(\ref{Eq:positive supremum cond})$ is clearly satisfied if there exists $\varphi$ such that $(\beta(\varphi),\varphi)=0$ and $(S_{d_{2}}\varphi,\varphi)>0$, which is easier to verify. If there is no $\varphi$ satisfying $(\ref{Eq:positive supremum cond})$ then the supremum in proofs of Theorems \ref{Thm:maximal eigenvalues of S-beta} and \ref{Thm:existence of maximum Neumann} is not positive. It follows that there is no positive eigenvalue of the operator $S_{d_2}-\beta$ in the case $\Gamma_D \neq\emptyset$ or of the problem $(\ref{Eq:reduced operator equation unilateral Neumann})$ in the case $\Gamma_D =\emptyset$. Indeed:
		\begin{itemize}
			\item{in the case $\Gamma_D \neq\emptyset$, if $d_1>0$ were an eigenvalue, then we would have $(S_{d_2}u,u)-(\beta(u),u)=d_{1}(u,u)>0$ for the corresponding eigenvector $u$, which would contradict the non-positivity of the supremum.}
			\item{in the case $\Gamma_D =\emptyset$, an eigenvector $u$ cannot be constant (see the end of the proof of Theorem \ref{Thm:existence of maximum Neumann}), and if $u$ were non-constant, then we would have $(S_{d_2}u,u)-(\beta(u),u)=d_1((I-A)u,u)>0$ by the last assertion of Remark \ref{Rem:eigenvalues of A}, which would be a contradiction again.}
		\end{itemize}	
		It follows that in the situation of Theorem \ref{Thm:tau estimate consequence} there exists $\varphi$ satisfying $(\ref{Eq:positive supremum cond})$ because that theorem guarantees the existence of a critical point in $D_U\cup C_{E}$ and consequently also existence of a positive eigenvalue of the operator $S_{d_2}-\beta$ in the case $\Gamma_D \neq\emptyset$ or of the problem $(\ref{Eq:reduced operator equation unilateral Neumann})$ in the case $\Gamma_D =\emptyset$ (see Sections \ref{Sec:Problem with Dirichlet and Neumann boundary conditions}, \ref{Sec:Problem with pure Neumann boundary conditions}).
	\end{remark}	
	\noindent The following theorem is formulated for both cases $\Gamma_{D}\neq\emptyset$ and $\Gamma_{D}=\emptyset$.
	\begin{theorem}\label{Thm:dMAXbeta < dMAX}		
		If $[d_{1},d_{2}]$ is a critical point of $(\ref{Eq:stationary problem with unilateral terms}),(\ref{Eq:zero boundary conditions})$, then always $d_{1}\leq d^{MAX}_{1}$.
		If $[d^{MAX}_{1}, d_{2}]\in C_{i}$ exactly for $i = j,\ldots,j+k-1$, all linear combinations $e$ of $e_{j},\ldots,e_{j+k-1}$ satisfy $(\ref{Eq:main thm cond})$ and $[d_{1},d_{2}]$ is a critical point of $(\ref{Eq:stationary problem with unilateral terms}),(\ref{Eq:zero boundary conditions})$, then $d_{1}<d^{MAX}_{1}$. Moreover, if the assumption $(\ref{Eq:positive supremum cond})$ is satisfied, then $d_{1}\leq d^{MAX,\beta}_{1}< d^{MAX}_{1}$.
	\end{theorem}

	The assumption concerning a position of $[d^{MAX}_{1}, d_{2}]$ is fulfilled either if $[d^{MAX}_{1}, d_{2}]$ lies in fact only on one hyperbola $C_j=...C_{j+k-1}$ (the eigenvalue $\kappa_j$ has the multiplicity $k$) or in the intersection of two different hyperbolas $C_j=...C_{j+l-1}\ne C_{j+l}=...C_{j+k-1}$ ($\kappa_j$ has the multiplicity $l$, $\kappa_{j+l}$ has the multiplicity $k-l$). See also Remark \ref{Rem:conciding hyperbolas}. Cf. also comments after Theorem \ref{Thm:General case}, where the assumptions are related to a set $C^{R}_{r}$, while in Theorem \ref{Thm:dMAXbeta < dMAX} they concern only one point $[d^{MAX}_{1},d_{2}]$ with a given fixed $d_{2}$.

	\begin{proof}\mbox{}\\
		First let's consider the case $\Gamma_{D}\neq\emptyset$.\\
		Let's show that if $(\ref{Eq:positive supremum cond})$ were fulfilled with no $\varphi$ then no critical point  of $(\ref{Eq:stationary problem with unilateral terms}),(\ref{Eq:zero boundary conditions})$ with $d_2$ under consideration would exist. If $[d_{1},d_{2}]$ were a critical point with $d_{1}>0$, then $d_{1}$ would be an eigenvalue of $S_{d_{2}}-\beta$ (see Lemma $\ref{Lem:critical points vs eigenvalues}$). Hence, we would have $u$ with $\lVert u\rVert_{H^{1}_{D}} = 1$ satisfying $(\ref{Eq:reduced operator equation unilateral})$. It would follow that $(S_{d_{2}}u,u)-(\beta(u),u)=d_1\lVert u\rVert_{H^{1}_{D}}>0$ and the condition $(\ref{Eq:positive supremum cond})$ would be satisfied with $\varphi=u$, which is a contradiction.\\
		Hence, in the following we can assume that $(\ref{Eq:positive supremum cond})$ is fulfilled with some $\varphi$.
		Due to Theorem $\ref{Thm:maximal eigenvalues of S-beta}$ and Lemma $\ref{Lem:properties of beta}$ we get
		\begin{equation*}
		d^{MAX,\beta}_{1} = \max\limits_{\substack{u\in H^{1}_{D} \\ u\neq o}}\frac{(S_{d_{2}}u,u)-(\beta(u),u)}{\lVert u\rVert^{2}_{H^{1}_{D}}} \leq \max\limits_{\substack{u\in H^{1}_{D} \\ u\neq o}}\frac{(S_{d_{2}}u,u)}{\lVert u\rVert^{2}_{H^{1}_{D}}} = d^{MAX}_{1}.
		\end{equation*}			
		As above, if $[d_{1},d_{2}]$ is a critical point $(\ref{Eq:stationary problem with unilateral terms}),(\ref{Eq:zero boundary conditions})$, then $d_{1}$ is an eigenvalue of $S_{d_{2}}-\beta$ (see Lemma $\ref{Lem:critical points vs eigenvalues}$). Hence, the first assertion of Theorem \ref{Thm:dMAXbeta < dMAX} is true.\\
		There exists $u_{0}\in H^{1}_{D}$ such that
		\begin{equation*}
		d^{MAX,\beta}_{1}= \frac{(S_{d_{2}}u_{0},u_{0})-(\beta(u_{0}),u_{0})}{\lVert u_{0}\rVert^{2}_{H^{1}_{D}}}.
		\end{equation*}	
		Due to Lemma \ref{Lem:properties of beta} we have
		\begin{equation}\label{estSbeta}
		d^{MAX,\beta}_{1}\leq \frac{(S_{d_{2}}u_{0},u_{0})}{\lVert u_{0}\rVert^{2}_{H^{1}_{D}}} \leq d^{MAX}_{1}.
		\end{equation}
		Let $(\beta(u_{0}),u_{0})=0$. Let us show that then the last inequality is strict. Indeed, if we had equality in (\ref{estSbeta}), then $u_{0}$ would be an eigenvector of $S_{d_{2}}$ corresponding to $d^{MAX}_{1}$, that means a linear combination of $e_{j},\ldots,e_{j+k-1}$ (see Remark \ref{Rem:sharing eigen of S}). Hence, $(\ref{Eq:main thm cond})$ with $e$ replaced by $u$ would be fulfilled by our assumptions and we would get $(\beta(u_0),u_0)>0$. This contradiction implies that the inequality in (\ref{estSbeta}) must be strict and we get $d^{MAX,\beta}_{1}<d^{MAX}_{1}$.\\
		If $(\beta(u_{0}),u_{0})>0$, then the first inequality in (\ref{estSbeta}) is strict and consequently $d^{MAX,\beta}_{1}<d^{MAX}_{1}$ again.\\
		\indent If $[d_{1},d_{2}]$ is a critical point of the problem $(\ref{Eq:stationary problem with unilateral terms}),(\ref{Eq:zero boundary conditions})$, then $d_{1}$ is an eigenvalue of $S_{d_{2}}-\beta$ by Lemma $\ref{Lem:critical points vs eigenvalues}$ and therefore
		\begin{equation*}
		d_{1}\leq d^{MAX,\beta}_{1}<d^{MAX}_{1}.
		\end{equation*}
		\indent The proof for $\Gamma_{D}=\emptyset$ is analogous, but we must use Remark \ref{Rem:eigenvalues of S} and Theorem \ref{Thm:existence of maximum Neumann}, in particular formula (\ref{Eq:maximal eigenvalue of S-beta Neumann}) instead of (\ref{Eq:maximal eigenvalue of S-beta}).
	\end{proof}

	\section{Proofs of main results}\label{Sec:proofs}
	We will use notation from the previous sections. 
	\subsection*{Proof of Theorem $\ref{Thm:General case}$}
	\indent i) Since $d_2>0$ was arbitrary in Section \ref{Sec:fixed d2} and $[d^{MAX}_{1},d_{2}]\in C_{E}$ (see Observation \ref{Ob:maximal eigenvalue of S}), it follows from Theorem \ref{Thm:dMAXbeta < dMAX} that there are no critical points of $(\ref{Eq:stationary problem with unilateral terms}),(\ref{Eq:zero boundary conditions})$ in $D_{S}$ (see also Figure \ref{Fig:hyperbolas}). Consequently there are also no bifurcation points of $(\ref{Eq:stationary problem nonlinear with unilateral terms}),(\ref{Eq:zero boundary conditions})$ in $D_S$ (see Lemma \ref{Lem:bifurcation implies critical} in Appendix).\\
	\indent ii) Let's consider the case $\Gamma_{D}\neq\emptyset$. Let's suppose the opposite, i.e. the assumptions of the second part of Theorem $\ref{Thm:General case}$ are satisfied and there are critical points of $(\ref{Eq:stationary problem with unilateral terms}),(\ref{Eq:zero boundary conditions})$ in $C^{R}_{r}(\varepsilon)$ for every $\varepsilon>0$. We can choose a sequence $d^{n}=[d^{n}_{1},d^{n}_{2}]\in D_{U}$ and $W_{n}=[u_{n},v_{n}]$ such that $d^{n}\rightarrow d^{0}\in C^{R}_{r}$, $\lVert W_{n}\rVert=\lVert u \rVert_{H^{1}_{D}} + \lVert v \rVert_{H^{1}_{D}}\neq0$ and $d^{n},W_{n}$ satisfy $(\ref{Eq:operator formulation unilateral})$. We can assume that $\frac{W_{n}}{\lVert W_{n}\rVert}\rightharpoonup W=[w,z]$. Let's divide $(\ref{Eq:operator formulation unilateral})$ by $\lVert W_{n}\rVert$ to get
	\begin{equation}\label{Eq:compact part convergence system}
	\begin{aligned}
	d^{n}_{1} \frac{u_{n}}{\lVert W_{n}\rVert} - b_{1,1}A\frac{u_{n}}{\lVert W_{n}\rVert} - b_{1,2}A\frac{v_{n}}{\lVert W_{n}\rVert} + \beta\left(\frac{u_{n}}{\lVert W_{n}\rVert}\right) &=& 0,\\
	d^{n}_{2} \frac{v_{n}}{\lVert W_{n}\rVert} - b_{2,1}A\frac{u_{n}}{\lVert W_{n}\rVert} - b_{2,2}A\frac{v_{n}}{\lVert W_{n}\rVert} &=& 0.	
	\end{aligned}		
	\end{equation}
	By the compactness of $A$ and $(\ref{Eq:beta property})$, we get $A\frac{u_{n}}{\lVert W_{n}\rVert}\rightarrow Aw$ and $\beta\left(\frac{u_{n}}{\lVert W_{n}\rVert}\right)\rightarrow \beta(w)$, analogously for $v_{n}$ and $z$. Hence, it follows easily from $(\ref{Eq:compact part convergence system})$ that $\frac{u_{n}}{\lVert W_{n}\rVert}\rightarrow w, \frac{v_{n}}{\lVert W_{n}\rVert}\rightarrow z$ and	\begin{eqnarray}
	d^{0}_{1} w - b_{1,1}Aw - b_{1,2}Az - \beta(w) &=& 0,\nonumber\\
	d^{0}_{2} z - b_{2,1}Aw - b_{2,2}Az &=& 0\nonumber.	
	\end{eqnarray}		
	Therefore the point $d^{0}=[d^{0}_{1},d^{0}_{2}]\in C^{R}_{r}$ is a critical point of the system $(\ref{Eq:stationary problem with unilateral terms}),(\ref{Eq:zero boundary conditions})$, which contradicts Theorem $\ref{Thm:dMAXbeta < dMAX}$ for $d_{2}=d^{0}_{2}$. Hence, there exists $\varepsilon>0$ such that there are no critical points of $(\ref{Eq:stationary problem with unilateral terms}),(\ref{Eq:zero boundary conditions})$ and consequently no bifurcation points of $(\ref{Eq:stationary problem nonlinear with unilateral terms}),(\ref{Eq:zero boundary conditions})$ in $C^{R}_{r}(\varepsilon)$ (see Lemma \ref{Lem:bifurcation implies critical} in Appendix).\\
	\indent The proof for $\Gamma_{D}=\emptyset$ is analogous, we only use the system $(\ref{Eq:operator formulation unilateral Neumann})$ instead of the system $(\ref{Eq:operator formulation unilateral})$.\\
	
	\begin{flushright}
		\qedsymbol
	\end{flushright}
	
	\subsection*{Proof of Theorem $\ref{Thm:Special case Dirichlet}$}
	\begin{enumerate}[i)]
		\item{Under the assumptions about $s_\pm$, either $e = e_{1}$ or $e = -e_{1}$ satisfies $s_{-}e^{-} - s_{+}e^{+} \equiv 0$. Since any point $d\in C_{1}$ is a critical point of the problem $(\ref{Eq:stationary problem}),(\ref{Eq:zero boundary conditions})$ with a non-trivial solution 
		$\left[\frac{d_{2}\kappa_{1}-b_{2,2}}{b_{2,1}}e_{1},e_{1}\right]$ due to Remark \ref{Rem:crit point <=> lies on C_j}, it is also a critical point of the problem $(\ref{Eq:stationary problem with unilateral terms}),(\ref{Eq:zero boundary conditions})$.}
		\item{Due to the definition of $d^I_2$, for $d_2<d^I_2$ we have $[d^{MAX}_1,d_2]\in C_j$ for a finite number of indices $j>1$ (see Section \ref{Sec:basic ass a def} and also Observation \ref{Ob:maximal eigenvalue of S}). 
		Any linear combination $e$ of the eigenfunctions $e_{j},j>1$ changes the sign (see Lemma $\ref{Lem:lin comb}$ in Appendix). Hence, under the assumptions of Theorem \ref{Thm:Special case Dirichlet} we have $s_{-}e^{-} - s_{+}e^{+} \not\equiv 0$. Therefore any critical point $[d_{1},d_{2}]$ with $d_{2}<d^{I}_{2}$ of the problem $(\ref{Eq:stationary problem with unilateral terms}),(\ref{Eq:zero boundary conditions})$ satisfies $d_{1} < d^{MAX}_{1}$ by Theorem $\ref{Thm:dMAXbeta < dMAX}$.	Now, it is possible to repeat the part ii) of the proof of Theorem $\ref{Thm:General case}$.}
	\end{enumerate}
	\begin{flushright}
		\qedsymbol
	\end{flushright}
	
	\subsection*{Proof of Theorem $\ref{Thm:Special case Neumann}$}
	For any $d_2>0$ we have $[d^{MAX}_1,d_2]\in C_j$ for a finite number of indices $j>0$ (see Section \ref{Sec:basic ass a def} and also Observation \ref{Ob:maximal eigenvalue of S}). Any linear combination $e$ of the eigenfunctions $e_{j}, j>0$ changes the sign (see Lemma $\ref{Lem:lin comb}$ for details), therefore the relation $s_{-}e^{-} - s_{+}e^{+} \not\equiv 0$ is always satisfied. Hence, any critical point $[d_{1},d_{2}]$ of the problem $(\ref{Eq:stationary problem with unilateral terms}),(\ref{Eq:zero boundary conditions})$ satisfies $d_{1} < d^{MAX}_{1}$ by Theorem $\ref{Thm:dMAXbeta < dMAX}$. Now, it is possible to repeat the part ii) of the proof of Theorem $\ref{Thm:General case}$.
	\begin{flushright}
		\qedsymbol
	\end{flushright}
	
	\subsection*{Proof of Theorem $\ref{Thm:tau estimate consequence}$}
	The assumption on $j_{0}$ directly implies that the $j_{0}$-th eigenvalue of the operator $S_{d_{2}}$ in the case $\Gamma_D \neq\emptyset$ or of the problem $(\ref{Eq:reduced operator equation Neumann})$ in the case $\Gamma_D =\emptyset$ is positive (see Remarks $\ref{Rem:sharing eigen of S}$, $\ref{Rem:eigenvalues of S}$ and $(\ref{Eq:hyperbolas}),(\ref{Eq:envelope})$). Hence, we have $(S_{d_{2}}e_{j_{0}},e_{j_{0}})>0$, where $e_{j_{0}}$ is the corresponding eigenvector. Let's denote $\tau := \max\left\{\lVert s_{-}\rVert_{\infty},\lVert s_{+}\rVert_{\infty}\right\}$. We get	
	\begin{equation*}
		\begin{aligned}
		(S_{d_{2}} e_{j_{0}},e_{j_{0}})-(\beta(e_{j_{0}}),e_{j_{0}}) &= (S_{d_{2}} e_{j_{0}},e_{j_{0}}) - \int_{\Omega} s_{+}e_{j_{0}}e^{+}_{j_{0}}\ d\Omega - \int_{\Omega} -s_{-}e_{j_{0}}e^{-}_{j_{0}}\ d\Omega \geq\\
		&\geq (S_{d_{2}} e_{j_{0}},e_{j_{0}}) - \lVert s_{+}\rVert_{\infty}\int_{\Omega} (e^{+}_{j_{0}})^{2} \ d\Omega - \lVert s_{-}\rVert_{\infty}\int_{\Omega} (e^{-}_{j_{0}})^{2} \ d\Omega\\
		&\geq (S_{d_{2}} e_{j_{0}},e_{j_{0}}) - \tau\left(\int_{\Omega} (e^{+}_{j_{0}})^{2} \ d\Omega + \int_{\Omega} (e^{-}_{j_{0}})^{2} \ d\Omega\right) \geq\\
		&\geq (S_{d_{2}} e_{j_{0}},e_{j_{0}}) - \tau (A e_{j_{0}},e_{j_{0}}).
		\end{aligned}		
	\end{equation*}	
	Since $e_{j_{0}}$ is non-trivial, we have $(A e_{j_{0}}, e_{j_{0}}) > 0$. Hence, if $\tau < \frac{(S_{d_{2}} e_{j_{0}},e_{j_{0}})}{(A e_{j_{0}},e_{j_{0}})}$, then $(S_{d_{2}} e_{j_{0}},e_{j_{0}})-(\beta(e_{j_{0}}),e_{j_{0}}) > 0$.\\ 
	\indent If $\Gamma_{D}\neq\emptyset$ then we get
	\begin{equation}\label{Eq:tau estimate D-N}
	\frac{(S_{d_{2}} e_{j_{0}},e_{j_{0}})}{(A e_{j_{0}},e_{j_{0}})} = \frac{\frac{1}{\kappa_{j_{0}}}(\frac{b_{1,2}b_{2,1}}{d_{2}\kappa_{j_{0}}-b_{2,2}} + b_{1,1})\lVert e_{j_{0}} \rVert^{2}_{H^{1}_{D}}}{\frac{1}{\kappa_{j_{0}}}\lVert e_{j_{0}} \rVert^{2}_{H^{1}_{D}}} = \frac{b_{1,2}b_{2,1}}{d_{2}\kappa_{j_{0}}-b_{2,2}} + b_{1,1}
	\end{equation}
	(see Remarks $\ref{Rem:sharing eigen of S}$ and $\ref{Rem:eigenvalues of A}$) and if $\Gamma_{D}=\emptyset$ we get
	\begin{equation}\label{Eq:tau estimate N}
	\frac{(S_{d_{2}} e_{j_{0}},e_{j_{0}})}{(A e_{j_{0}},e_{j_{0}})} = \frac{\frac{1}{\kappa_{j_{0}}+1}(\frac{b_{1,2}b_{2,1}}{d_{2}\kappa_{j_{0}}-b_{2,2}} + b_{1,1})\lVert e_{j_{0}} \rVert^{2}_{H^{1}_{D}}}{\frac{1}{\kappa_{j_{0}}+1}\lVert e_{j_{0}} \rVert^{2}_{H^{1}_{D}}} = \frac{b_{1,2}b_{2,1}}{d_{2}\kappa_{j_{0}}-b_{2,2}} + b_{1,1}.
	\end{equation}
	(see Remarks $\ref{Rem:eigenvalues of S}$ and $\ref{Rem:eigenvalues of A}$). Hence, if $\tau < \frac{b_{1,2}b_{2,1}}{d_{2}\kappa_{j_{0}}-b_{2,2}} + b_{1,1}$, then the assumption $(\ref{Eq:positive supremum cond})$ of Theorems \ref{Thm:maximal eigenvalues of S-beta},\ref{Thm:existence of maximum Neumann} is satisfied with $\varphi=e_{j_0}$ and therefore $d^{MAX,\beta}_{1}>0$ exists. A point $[d^{MAX,\beta}_{1},d_{2}]$ is a critical point of $(\ref{Eq:stationary problem with unilateral terms}),(\ref{Eq:zero boundary conditions})$ by Lemma $\ref{Lem:critical points vs eigenvalues}$ and it lies in $D_{U}\cup C_{E}$ by Theorem \ref{Thm:dMAXbeta < dMAX} and because $[d^{MAX}_{1},d_{2}]\in C_{E}$ (see Observation \ref{Ob:maximal eigenvalue of S}).
	\begin{flushright}
		\qedsymbol
	\end{flushright}
	
	\subsection*{Proof of Theorem \ref{Thm:General case source on boundary}}		
	If unilateral terms in boundary conditions are considered, we replace the operators $\tilde{F}$ and $\beta$ in $(\ref{Eq:operator formulation unilateral})$,$(\ref{Eq:operator formulation unilateral nonlinear})$,$(\ref{Eq:operator formulation unilateral Neumann})$,$(\ref{Eq:operator formulation unilateral noonlinear Neumann})$ by the operator $\beta_{N}$, which has the same properties as $\beta$. Then it is necessary to repeat whole Section $\ref{Sec:fixed d2}$ for this operator. The actual proof of Theorem $\ref{Thm:General case source on boundary}$ is then the same as the proof of Theorem $\ref{Thm:General case}$.
	
	\setcounter{section}{0}
	\renewcommand{\thesection}{\Alph{section}}
	\section{Appendix}
	For a completeness we give here proofs of two standard assertions used in the text (Lemmas \ref{Lem:lin comb}, \ref{Lem:bifurcation implies critical}), and a slightly simplified proof of a result given already in \cite{Kucera-Navratil-Dirichlet} (Lemma \ref{Lem:Navratil conditions}).
	\begin{lemma}\label{Lem:lin comb}
		Any linear combination $\sum_{k = j_{0}}^{n} a_{k} e_{k},n\in\mathbb{N}$ of eigenfunctions of $(\ref{Eq:Laplace eigenvalue problem})$, where $j_{0}=2$ for $\Gamma_{D}\neq\emptyset$ and $j_{0}=1$ for $\Gamma_{D}=\emptyset$, changes the sign on the domain $\Omega$.
	\end{lemma}
	\begin{proof}\mbox{}\\
		\textbf{Case $\Gamma_{D}=\emptyset$:}\newline
		Since $e_{k}$ is an orthonormal basis (see Remark $\ref{Rem:eigenvalues of A}$) and $e_{0}$ is constant, we have
		\begin{equation*}
		\int_{\Omega} e_{0}\sum_{k = j_{0}}^{n} a_{k} e_{k}\; d\Omega =  \sum_{k = j_{0}}^{n} a_{k} \int_{\Omega} e_{0}e_{k}\; d\Omega = \sum_{k = j_{0}}^{n} a_{k} \int_{\Omega} \left(\nabla e_{0}\nabla e_{k} + e_{0}e_{k}\right)\; d\Omega = \sum_{k = j_{0}}^{n} a_{k} (e_{0},e_{k})_{W^{1,2}} = 0.
		\end{equation*}
		Hence, the function $e_{0}\sum_{k = j_{0}}^{n} a_{k} e_{k}$ changes the sign on the domain $\Omega$. Since $e_{0}$ is constant, also the function $\sum_{k = j_{0}}^{n} a_{k} e_{k}$ changes the sign on $\Omega$.\\
		\textbf{Case $\Gamma_{D}\neq\emptyset$:}\newline
		We will use the eigenfunction $e_{1}$ instead of $e_{0}$. Again since $e_{k}$ is the orthonormal basis, we have
		\begin{equation*}
		\int_{\Omega} e_{1}e_{k}\; d\Omega = \int_{\Omega} -\frac{1}{\kappa_{1}}\Delta e_{1} e_{k}\; d\Omega = \frac{1}{\kappa_{1}}\int_{\Omega} \nabla e_{1} \nabla e_{k}\; d\Omega = \frac{1}{\kappa_{1}}(e_{1},e_{k})_{H^{1}_{D}} = 0\quad\text{ for any }k>1.
		\end{equation*}
		The rest is the same as in the case $\Gamma_{D}=\emptyset$.
	\end{proof}
	\begin{lemma}\label{Lem:bifurcation implies critical}
		Every bifurcation point $[d_{1},d_{2}]$ of $(\ref{Eq:stationary problem nonlinear with unilateral terms}),(\ref{Eq:zero boundary conditions})$ is also a critical point of $(\ref{Eq:stationary problem with unilateral terms}),(\ref{Eq:zero boundary conditions})$.
	\end{lemma}
	\begin{proof}\mbox{}\\
		We will show the proof for $\Gamma_{D}\neq\emptyset$. The proof for $\Gamma_{D}=\emptyset$ is the same, we only use the system $(\ref{Eq:operator formulation unilateral noonlinear Neumann})$ instead of the system $(\ref{Eq:operator formulation unilateral nonlinear})$.\\
		\indent Let $d^{0}=[d_{1},d_{2}]\in\mathbb{R}^{2}_{+}$ be a bifurcation point of $(\ref{Eq:stationary problem nonlinear with unilateral terms}),(\ref{Eq:zero boundary conditions})$. Then there exists a sequence $d^{n}=[d^{n}_{1},d^{n}_{2}]$ such that $d^{n}\rightarrow d^{0}$ and $W_{n}=[u_{n},v_{n}]\rightarrow 0$ with $\lVert W_{n}\rVert=\lVert u \rVert_{H^{1}_{D}} + \lVert v \rVert_{H^{1}_{D}}\neq0$ and $d^{n},W_{n}$ satisfy $(\ref{Eq:stationary problem nonlinear with unilateral terms}),(\ref{Eq:zero boundary conditions})$, i.e. $(\ref{Eq:operator formulation unilateral nonlinear})$. We can assume $\frac{W_{n}}{\lVert W_{n}\rVert}\rightharpoonup W=[w,z]$. Let's divide the system $(\ref{Eq:operator formulation unilateral nonlinear})$ by $\lVert W_{n}\rVert$. We get
		\begin{equation}\label{Eq:bif crit lemma convergence}
		\begin{aligned}
		d^{n}_{1} \frac{u_{n}}{\lVert W_{n}\rVert} - b_{1,1}A\frac{u_{n}}{\lVert W_{n}\rVert} - b_{1,2}A\frac{v_{n}}{\lVert W_{n}\rVert} - \frac{N_{1}(u_{n},v_{n})}{\lVert W_{n}\rVert} +\frac{\tilde{F}(u_{n})}{\lVert W_{n}\rVert} &=& 0,\\
		d^{n}_{2} \frac{v_{n}}{\lVert W_{n}\rVert} - b_{2,1}A\frac{u_{n}}{\lVert W_{n}\rVert} - b_{2,2}A\frac{v_{n}}{\lVert W_{n}\rVert} - \frac{N_{2}(u_{n},v_{n})}{\lVert W_{n}\rVert} &=& 0
		\end{aligned}	
		\end{equation}
		due to linearity of $A$. Due to $(\ref{Eq:nonlinear operator limit condition})$ we have $\frac{N_{j}(u_{n},v_{n})}{\lVert W_{n}\rVert}\rightarrow 0$ as $n\rightarrow+\infty$ for $j=1,2$. Since $\frac{u_{n}}{\lVert W_{n}\rVert}\rightharpoonup w$ and $\frac{v_{n}}{\lVert W_{n}\rVert}\rightharpoonup z$, using compactness of $A$ and $(\ref{Eq:beta F relation})$ we get $A\frac{u_{n}}{\lVert W_{n}\rVert}\rightarrow Aw$ and $\frac{\tilde{F}(u_{n})}{\lVert W_{n}\rVert}\rightarrow \beta(w)$, analogously for $v_n$ and $z$. We have $d_{1},d_{2}>0$, therefore it follows from $(\ref{Eq:bif crit lemma convergence})$ that $\frac{u_{n}}{\lVert W_{n}\rVert}\rightarrow w, \frac{v_{n}}{\lVert W_{n}\rVert}\rightarrow z$, $\lVert W\rVert = 1$ and
		\begin{eqnarray}
		d^{0}_{1} w - b_{1,1}Aw - b_{1,2}Az + \beta(w) &=& 0,\nonumber\\
		d^{0}_{2} z - b_{2,1}Aw - b_{2,2}Az &=& 0\nonumber.	
		\end{eqnarray}		
		Therefore the point $d^{0}$ is a critical point of the system $(\ref{Eq:stationary problem with unilateral terms}),(\ref{Eq:zero boundary conditions})$.
	\end{proof}
	\begin{lemma}[see \cite{Kucera-Navratil-Dirichlet}]\label{Lem:Navratil conditions}
		For any $d_{2}>0$ and $u_{0}\in H^{1}_{D}$ the operators $P\equiv\beta$ and $P\equiv S_{d_{2}}-\beta$ satisfy the condition $(\ref{Eq:Navratil limit condition})$.
	\end{lemma}
	\begin{proof}\mbox{}\\
		We will prove $(\ref{Eq:Navratil limit condition})$ for $P=\beta^{-}$. The proof for $\beta^{+}$ is analogous and for $P\equiv S_{d_{2}}-\beta$ it will follow by using the definition of $\beta$ and linearity of $S_{d_2}$.	\\		
		Let $u_0 ,h\in H^{1}_{D}$. We will introduce two sets $\Omega_{0}$ and $\Omega_{th}$ such that
		\begin{equation*}
		\begin{aligned}
		u_{0}(\mathbf{x})  < 0\text{ a.e. on }\Omega_{0},&\quad u_{0}(\mathbf{x}) \geq 0\text{ a.e. on }\Omega\setminus\Omega_{0},\\
		u_{0}(\mathbf{x}) +th(\mathbf{x}) < 0\text{ a.e. on }\Omega_{th},&\quad u_{0}(\mathbf{x}) +th(\mathbf{x}) \geq 0\text{ a.e. on }\Omega\setminus\Omega_{th}.
		\end{aligned}
		\end{equation*}		
		Then
		\begin{equation*}
		\begin{aligned}
		\frac{1}{t}(\beta^{-}(u_{0}+th)-\beta^{-}(u_{0}),u_{0})& -(\beta^{-}(u_{0}),h) =\\
		&= \frac{1}{t}\left[\int_{\Omega}-(u_{0}+th)^{-}u_{0} + u_{0}u^{-}_{0}\, d\Omega  \right] -\int_{\Omega} -u^{-}_{0}h \, d\Omega =\\
		&= \frac{1}{t}\left[\int_{\Omega_{th}} (u_{0}+th)u_{0}\, d\Omega_{th} - \int_{\Omega_{0}}  u^{2}_{0}\, d\Omega_{0} \right] -\int_{\Omega_{0}} u_{0}h \, d\Omega_{0} =\\
		&= \frac{1}{t}\left[\int_{\Omega_{th}} u^{2}_{0}\, d\Omega_{th} - \int_{\Omega_{0}}  u^{2}_{0}\, d\Omega_{0} \right] + \int_{\Omega_{th}} u_{0}h\,  d\Omega_{th} -\int_{\Omega_{0}} u_{0}h \, d\Omega_{0}.
		\end{aligned}
		\end{equation*}
		We can afford to work with the definition of $\beta^{-}$ without $s_{-}$, because it is non-negative, i.e. it does not affect the sign of terms under integration.\\
		Let $\chi_{th}$ and $\chi_{0}$ be the characteristic function of $\Omega_{th}$ and $\Omega_{0}$, respectively. We have
		\begin{equation*}
		\lim\limits_{t\to 0}\int_{\Omega_{th}} u_{0}h\,d\Omega_{th} = \lim\limits_{t\to 0}\int_{\Omega} u_{0}h\chi_{th}\,d\Omega = \int_{\Omega} u_{0}h\chi_{0}\,d\Omega = \int_{\Omega_{0}} u_{0}h\,d\Omega_{0}
		\end{equation*}
		by Dominated Convergence theorem. Let's introduce sets $\Omega_{th1},\Omega_{th2},\Omega_{th3}$ such that	
		\begin{equation*}
		\begin{aligned}
		u_{0}(\mathbf{x}) < -t h(\mathbf{x})\text{ and }u_{0}(\mathbf{x})< 0\text{ almost everywhere on }\Omega_{th1},\\
		u_{0}(\mathbf{x}) < -t h(\mathbf{x})\text{ and }u_{0}(\mathbf{x})\geq 0\text{ almost everywhere on }\Omega_{th2},\\
		u_{0}(\mathbf{x}) \geq -t h(\mathbf{x})\text{ and }u_{0}(\mathbf{x})< 0\text{ almost everywhere on }\Omega_{th3},
		\end{aligned}
		\end{equation*}
		with $\Omega_{th} = \Omega_{th1} \cup \Omega_{th2}$ and $\Omega_{0} = \Omega_{th1} \cup \Omega_{th3}$. This way we get
		\begin{equation*}
		\begin{aligned}
		\int_{\Omega_{th}} u^{2}_{0}\,d\Omega_{th} - \int_{\Omega_{0}} u^{2}_{0}\,d\Omega_{0} &= \int_{\Omega_{th1}} u^{2}_{0}\,d\Omega_{th1} + \int_{\Omega_{th2}} u^{2}_{0}\,d\Omega_{th2} - \int_{\Omega_{th1}} u^{2}_{0}\,d\Omega_{th1} - \int_{\Omega_{th3}} u^{2}_{0}\,d\Omega_{th3} =\\
		&= \int_{\Omega_{th2}} u^{2}_{0}\,d\Omega_{th2} - \int_{\Omega_{th3}} u^{2}_{0}\,d\Omega_{th3}.
		\end{aligned}		
		\end{equation*}
		Since $0\leq u_{0} < -t h$ a.e. on $\Omega_{th2}$ and $0 > u_{0} \geq -th$ a.e. on $\Omega_{th3}$, we get
		\begin{equation*}
		\lim\limits_{t\to 0} \frac{1}{t}\left( \int_{\Omega_{th2}} u^{2}_{0}\,d\Omega_{th2} - \int_{\Omega_{th3}} u^{2}_{0}\,d\Omega_{th3}\right) \leq \lim\limits_{t\to 0} \frac{1}{t}  \left(\int_{\Omega_{th2}} (th)^{2}\,d\Omega_{th2} - \int_{\Omega_{th3}} (th)^{2}\,d\Omega_{th3}\right) = 0.
		\end{equation*}
		Hence, it follows from the discussion above that
		\begin{equation*}
		\lim\limits_{t\to 0} \frac{1}{t}(\beta^{-}(u_{0}+th) -\beta^{-}(u_{0}),u_{0})-(\beta^{-}(u_{0}),h) =0
		\end{equation*}
		which proves $(\ref{Eq:Navratil limit condition})$ for $\beta^{-}$. 		
	\end{proof}

	\setcounter{section}{3}
	\renewcommand{\thesection}{\arabic{section}}
	\section{Acknowledgement}
	M.Fencl has been supported by the project SGS-2016-003 of University of West Bohemia and the project LO1506 of the Czech Ministry of Education, Youth and Sport. M. Ku\v cera has been supported by RVO:67985840.

\end{document}